\documentclass[12pt,reqno]{amsart}

\usepackage{amssymb,graphicx, amsmath, amsthm, MnSymbol, array}
\usepackage[all]{xy}

\usepackage{tikz-cd}

\newtheorem{theorem}{Theorem}[section]
\newtheorem{lemma}[theorem]{Lemma}
\newtheorem{proposition}[theorem]{Proposition}
\newtheorem{corollary}[theorem]{Corollary}

\theoremstyle{definition}
\newtheorem{definition}[theorem]{Definition}
\newtheorem{example}[theorem]{Example}

\newtheorem{remark}[theorem]{Remark}
\newtheorem{convention}[theorem]{Convention}

\def\bal{\boldsymbol{\mathit{ba}\ell}}

\def\ubal{\boldsymbol{\mathit{uba}\ell}}
\def\dbal{\boldsymbol{\mathit{dba}\ell}}

\def\cubal{\boldsymbol{\mathit{cuba}\ell}}

\def\balg{\boldsymbol{\mathit{balg}}}

\def\basic{\boldsymbol{\mathit{basic}}}
\def\ubasic{\boldsymbol{\mathit{ubasic}}}
\def\mbasic{\boldsymbol{\mathit{mbasic}}}

\def\Set{{\sf Set}}
\def\CABA{{\sf CABA}}
\def\At{{\sf At}}

\newcommand{\func}[1]{\operatorname{#1}}

\def\int{{\sf int}}

\newcommand\prim{\func{Prim}}

\def\Id{\func{Id}}

\newcommand\C{{\sf Comp}}
\newcommand\SC{{\sf SComp}}
\newcommand\creg{{\sf CReg}}

\newcommand\KHaus{{\sf KHaus}}

\begin{document}

\title[A generalization of Gelfand-Naimark-Stone duality]{A generalization of Gelfand-Naimark-Stone duality to completely regular spaces}
\author{G.~Bezhanishvili}
\address{New Mexico State University}
\email{guram@nmsu.edu}

\author{P.~J.~Morandi}
\address{New Mexico State University}
\email{pmorandi@nmsu.edu}

\author{B.~Olberding}
\address{New Mexico State University}
\email{bruce@nmsu.edu}
\date{}


\begin{abstract}
Gelfand-Naimark-Stone duality establishes a dual equivalence between the category $\KHaus$ of compact Hausdorff spaces
and the category $\ubal$ of uniformly complete bounded archimedean $\ell$-algebras. We extend this duality to the category
$\creg$ of completely regular spaces. This we do by first introducing basic extensions of bounded archimedean $\ell$-algebras
and generalizing Gelfand-Naimark-Stone duality to a dual equivalence between the category $\ubasic$ of
uniformly complete basic extensions and the category $\C$ of compactifications of completely regular spaces. We then introduce
maximal basic extensions and prove that the subcategory $\mbasic$ of $\ubasic$ consisting of
maximal basic extensions is dually equivalent to the subcategory $\SC$ of $\C$ consisting of Stone-\v{C}ech compactifications.
This yields the desired dual equivalence for completely regular spaces since $\creg$ is equivalent to $\SC$.
\end{abstract}

\maketitle

\section{Introduction}

Let $\creg$ be the category of completely regular spaces and continuous maps, and let $\KHaus$ be its full subcategory consisting of
compact Hausdorff spaces. Let also $\bal$ be the category of bounded archimedean $\ell$-algebras and unital $\ell$-algebra homomorphisms,
and let $\ubal$ be its full subcategory consisting of uniformly complete objects in $\bal$ (see Section~2 for definitions). There is a
contravariant functor $C^*:\creg\to\bal$ sending a completely regular space $X$ to the $\ell$-ring $C^*(X)$ of bounded continuous real-valued
functions, and a contravariant functor $Y:\bal\to\creg$ sending $A\in\bal$ to the space of maximal $\ell$-ideals. The functors $C^*$ and
$Y$ define a contravariant adjunction between $\creg$ and $\bal$ such that $C^*(X) \in \ubal$ for each $X \in \KHaus$ and
$Y(A) \in \KHaus$ for each $A \in \bal$.
Thus, the contravariant adjunction between $\creg$ and $\bal$ restricts to a dual equivalence between $\KHaus$ and $\ubal$. This dual
equivalence is known as Gelfand-Naimark-Stone duality (see \cite{GN43,Sto40,GJ60,Joh82,BMO13a}). We note that if $X\in\KHaus$, then
every continuous real-valued function on $X$ is bounded. Therefore, $C^*(X)$ is equal to the $\ell$-ring $C(X)$ of all continuous
real-valued functions. Thus, the functor $C^*:\creg\to\bal$ restricts to the functor $C:\KHaus\to\ubal$, and we arrive at the
following commutative diagram.
\[
\begin{tikzcd}[column sep=5pc, row sep = 1pc]
\creg \arrow[r, shift left = .5ex, "C^*"] & \bal \arrow[l, shift left = .5ex, "Y"] \\
\KHaus \arrow[u, hookrightarrow] \arrow[r, shift left = .5ex, "C"] & \ubal \arrow[u, hookrightarrow] \arrow[l, shift left = .5ex, "Y"]
\end{tikzcd}
\]

The purpose of this article is to extend Gelfand-Naimark-Stone duality to completely regular spaces. For this it is not
sufficient to only work with the $\ell$-ring $C^*(X)$. The space of maximal $\ell$-ideals of $C^*(X)$ is the
Stone-\v{C}ech compactification $\beta X$, and hence $X$ is not recoverable as the space of maximal $\ell$-ideals of $C^*(X)$.
To recover $X$ additional data is required, which we show
can be provided by also working with the $\ell$-ring $B(X)$ of bounded functions on $X$. The idempotents of $B(X)$
are exactly the characteristic functions of subsets of $X$, so the boolean algebra $\Id(B(X))$ of idempotents of $B(X)$ is isomorphic
to the powerset $\wp(X)$. The singletons $\{x\}$ are the atoms of $\wp(X)$, which correspond to the primitive idempotents of $B(X)$.
Therefore, $X$ is in bijective correspondence with the primitive idempotents of $B(X)$.
Thus, to recover the topology on $X$ it is sufficient to give an algebraic description of the embedding $C^*(X)\to B(X)$.
Since $C^*(X)$ is isomorphic to $C(\beta X)$, it suffices to give an algebraic description of the monomorphism $C(\beta X)\to B(X)$
arising from the embedding $X\to\beta X$. More generally, given a compactification $e:X\to Y$, we will give an algebraic description
of the monomorphism $C(Y)\to B(X)$ arising from $e$ (it is a monomorphism since $e[X]$ is dense in $Y$).

For this we will first characterize the algebras $B(X)$ as Dedekind complete (bounded) archimedean $\ell$-algebras in which the
boolean algebra of idempotents is atomic. We term such algebras \emph{basic algebras}, and prove that the category $\balg$ of
basic algebras and the unital $\ell$-algebra homomorphisms between them that are \emph{normal} (meaning that they preserve all existing
joins, and hence meets) is dually equivalent to the category $\sf Set$ of sets and functions. This provides a ring-theoretic version
of Tarski duality between $\sf Set$ and the category $\sf CABA$ of complete and atomic boolean algebras and complete boolean
homomorphisms.

We next extend the focus from algebras in $\bal$ to what we call \emph{basic extensions}. These are extensions
$\alpha:A \to B$ such that $A\in\bal$, $B\in\balg$, and $\alpha[A]$ is join-meet dense in $B$.
Each compactification $e:X\to Y$ gives rise to the basic extension $e^\flat : C(Y)\to B(X)$. In Theorem~\ref{duality cptf}
we prove that this correspondence extends to a dual adjunction between the category $\C$ of compactifications and the
category $\basic$ of basic extensions, which restricts to a dual equivalence between $\C$ and the full subcategory
$\ubasic$ of $\basic$ consisting of \emph{uniformly complete basic extensions}.

We further consider the full subcategory $\sf SComp$ of $\C$ consisting of Stone-\v{C}ech compactifications, and prove
that the dual equivalence between $\C$ and $\ubasic$ restricts to a dual equivalence between $\sf SComp$ and
the full subcategory $\mbasic$ of $\ubasic$ consisting of \emph{maximal basic extensions},
which can be characterized as
those uniformly complete basic extensions $\alpha:A\to B$ for which the only elements of $B$ that are both a join and meet
of elements from $\alpha[A]$ are the elements of $\alpha[A]$ itself.
Since the category $\sf CReg$ of completely regular spaces is equivalent to $\sf SComp$, we conclude
that the maximal basic extensions provide an algebraic counterpart of the completely regular spaces.

This article can be viewed as a ring-theoretic companion to our article \cite{BMO19a}, in which we show that the category of completely regular
spaces is dually equivalent to the category of what we call maximal de Vries extensions in \cite{BMO19a}, a certain class of
extensions of complete Boolean algebras equipped with a proximity relation. In so doing we extend de Vries duality to completely regular spaces
in direct analogy with how we extend Gelfand-Naimark-Stone duality to completely regular spaces in the present paper. In a future paper,
we will make the analogy between these two settings more precise. For the present, however, these two dualities for completely regular spaces
remain independent of each other in our approaches.

The article is organized as follows.
In Section 2 we recall Gelfand-Naimark-Stone duality and describe its restriction to the full subcategories of $\KHaus$ consisting of
Stone spaces and extremally disconnected spaces.
Section 3 introduces basic algebras and their fundamental properties. We prove that $\balg$ is dually equivalent to $\sf Set$, which is a
ring-theoretic version of Tarski duality between $\sf CABA$ and $\sf Set$.
In Section 4 we define basic extensions and uniformly complete basic extensions, and show that the category $\ubasic$ of uniformly
complete basic extensions is a reflective subcategory of the category $\basic$ of basic extensions.
We also define a functor from $\C$ to $\basic$, and show that it lands in $\ubasic$.
In Section 5 we produce a functor going the other way, from $\basic$ to $\C$.
With these functors in place, we show in Section 6 that there is a dual adjunction between $\basic$ and $\C$, which
restricts to a dual equivalence between $\ubasic$ and $\C$.
Finally, building on the previous sections, we obtain in Section~7 our generalization of Gelfand-Naimark-Stone duality between the category
$\sf CReg$ of completely regular spaces and the category $\mbasic$ of maximal basic extensions, a special class of basic extensions that
we describe in detail in Section 7.

\section{Gelfand-Naimark-Stone duality}

In this section we recall Gelfand-Naimark-Stone duality. This requires recalling a number of basic facts about ordered rings and algebras.
For general references we use \cite{Bir79, GJ60, Joh82, LZ71}. For a detailed study of the category $\bal$, which plays a central role for
our purposes, we refer to \cite{BMO13a}.

For a completely regular space $X$, let $C(X)$ be the ring of continuous real-valued functions, and let $C^*(X)$ be the subring of $C(X)$
consisting of bounded functions. We note that if $X$ is compact, then $C^*(X)=C(X)$. There is a natural partial order $\le$ on $C(X)$ lifted
from $\mathbb R$. Then $C^*(X)$ with the restriction of $\le$ is a bounded archimedean $\ell$-algebra, where we recall that
\begin{itemize}
\item A ring $A$ with a partial order $\le$ is an \emph{$\ell$-ring} (lattice-ordered ring) if $(A,\le)$ is a lattice, $a\le b$ implies
$a+c \le b+c$ for each $c$, and $0 \leq a, b$ implies $0 \le ab$.
\item An $\ell$-ring $A$ is \emph{bounded} if for each $a \in A$ there is $n \in \mathbb{N}$ such that $a \le n\cdot 1$ (that is, $1$ is
a \emph{strong order unit}).
\item An $\ell$-ring $A$ is \emph{archimedean} if for each $a,b \in A$, whenever $na \le b$ for each $n \in \mathbb{N}$, then $a \le 0$.
\item An $\ell$-ring $A$ is an \emph{$\ell$-algebra} if it is an $\mathbb R$-algebra and for each $0 \le a\in A$ and $0\le r\in\mathbb R$ we
have $ra\ge 0$.
\item Let $\bal$ be the category of bounded archimedean $\ell$-algebras and unital $\ell$-algebra homomorphisms.
\end{itemize}

\begin{convention}
For a continuous map $\varphi:X\to Y$ between completely regular spaces let $\varphi^* : C^*(Y)\to C^*(X)$ be given by $\varphi^*(f) = f \circ \varphi$.
\end{convention}
Then $\varphi^*$ is a unital $\ell$-algebra homomorphism, and we have a contravariant functor $C^* : \creg \to \bal$ which sends each $X\in\creg$ to the $\ell$-algebra $C^*(X)$, and each continuous map $\varphi:X\to Y$ to the unital $\ell$-algebra
homomorphism $\varphi^* : C^*(Y)\to C^*(X)$. We denote the restriction of $C^*$ to $\KHaus$ by $C$ since for $X\in\KHaus$ we have $C^*(X)=C(X)$.

The functor $C^*$ has a contravariant adjoint which is defined as follows.
For $A\in\bal$ and $a\in A$, we recall that
the \emph{absolute value} of $a$ is defined as $|a|=a\vee(-a)$, that an ideal $I$ of $A$ is an \emph{$\ell$-ideal} if $|a|\le|b|$ and
$b\in I$ imply $a\in I$, and that $\ell$-ideals are exactly the kernels of $\ell$-algebra homomorphisms. Let $Y_A$ be the space of
maximal $\ell$-ideals of $A$, whose closed sets are exactly sets of the form
\[
Z_\ell(I) = \{M\in Y_A\mid I\subseteq M\},
\]
where $I$ is an $\ell$-ideal of $A$. The space $Y_A$ is often referred to as the \emph{Yosida space} of $A$, and it is well known that
$Y_A\in\KHaus$.

\begin{convention}
For a unital $\ell$-algebra homomorphism $\alpha:A\to B$ let $\alpha_*:Y_B\to Y_A$ be
given by $\alpha_*(M)=\alpha^{-1}(M)$.
\end{convention}

Then $\alpha_*$ is continuous, and we have a contravariant functor $Y:\bal\to\creg$ which sends each $A\in\bal$ to the
compact Hausdorff space $Y_A$, and each unital $\ell$-algebra homomorphism $\alpha:A\to B$ to the continuous map $\alpha_*:Y_B\to Y_A$.

For $A\in\bal$ and $X\in\creg$, we have $\mathrm{hom}_{\bal}(A,C^*(X))\simeq\mathrm{hom}_{\creg}(X,Y_A)$. Thus, $Y$ and $C^*$ define a
contravariant adjunction between $\bal$ and $\creg$.
As we already pointed out, $Y_A \in \KHaus$. In fact, each compact Hausdorff space is homeomorphic to $Y_A$ for some $A \in \bal$. To see this, for $X\in\creg$,
associate with each $x\in X$ the maximal $\ell$-ideal
\[
M_x := \{ f \in C^*(X) \mid f(x) = 0 \}.
\]
Then $\xi_X:X\to Y_{C^*(X)}$ given by $\xi_X(x)=M_x$ is an embedding, and it is a homeomorphism iff $X$ is compact.

To describe which objects of $\bal$ are isomorphic to $C^*(X)$ for some $X$, we observe that for a maximal $\ell$-ideal $M$ of $A\in\bal$, we have $A/M\simeq\mathbb R$. Therefore, with each $a\in A$,
we can associate $f_a\in C(Y_A)$ given by $f_a(M)=a+M$. Then $\zeta_A:A\to C(Y_A)$ given by $\zeta_A(a)=f_a$ is a unital $\ell$-algebra
homomorphism, which is a monomorphism since the intersection of maximal $\ell$-ideals is 0. It is an isomorphism iff the norm on $A$ defined by
\[
||a||=\inf\{r\in\mathbb R\mid |a|\le r\}
\]
is complete. In such a case we call $A$ \emph{uniformly complete}, and denote the full subcategory of $\bal$ consisting of uniformly complete
$\ell$-algebras by $\ubal$. Thus, $A\in\ubal$ iff $A$ is isomorphic to $C^*(X)$ for some $X\in\KHaus$.
Consequently, the contravariant adjunction $(C^*,Y)$ between $\creg$ and
$\bal$ restricts to a dual equivalence $(C,Y)$ between $\KHaus$ and $\ubal$, and
we arrive at the following celebrated result:

\begin{theorem} [Gelfand-Naimark-Stone duality]
The categories $\KHaus$ and $\ubal$ are dually equivalent, and the dual equivalence is established by the functors $C$ and $Y$.
\end{theorem}

\begin{remark}
While Stone worked with real-valued functions, Gelfand and Naimark worked with complex-valued functions.
As a result, Gelfand-Naimark duality is between $\KHaus$ and the category ${\sf CC}^*{\sf Alg}$ of commutative
$C^*$-algebras. However, $\ubal$ is equivalent to ${\sf CC}^*{\sf Alg}$, and the equivalence is established by the
following functors. The functor $\ubal\to{\sf CC}^*{\sf Alg}$ associates to each $A\in\ubal$ its complexification
$A\otimes_\mathbb R \mathbb C$; and the functor ${\sf CC}^*{\sf Alg}\to\ubal$ associates to each $A\in{\sf CC}^*{\sf Alg}$
its subalgebra of self-adjoint elements; for further details see, e.g., \cite[Sec.~7]{BMO13a}.
\end{remark}

We recall that a subset of a topological space $X$ is \emph{clopen} if it is closed and open, that $X$ is \emph{zero-dimensional}
if it has a basis of clopens, and that $X$ is \emph{extremally disconnected} if the closure of each open is clopen.
Zero-dimensional compact Hausdorff spaces are usually referred to as \emph{Stone spaces} because, by the celebrated Stone duality,
they provide the dual counterpart of boolean algebras.

Let $\sf Stone$ be the full subcategory of $\KHaus$ consisting of Stone spaces, and let $\sf ED$ be the full subcategory of
$\sf Stone$ consisting of extremally disconnected objects of $\KHaus$. Gelfand-Naimark-Stone duality yields interesting
restrictions to $\sf Stone$ and $\sf ED$.

We recall that a commutative ring $A$ is a \emph{clean ring} provided each element of $A$ is the sum of an idempotent and a unit.
Let $\cubal$ be the full subcategory of $\ubal$ consisting of clean rings. By \cite[Thm.~2.5]{Aza02}, a compact Hausdorff space $X$ is a
Stone space iff $C(X)$ is a clean ring. This together with Gelfand-Naimark-Stone duality yields:

\begin{corollary}
The categories $\sf Stone$ and $\cubal$ are dually equivalent, and the dual equivalence is established by restricting the functors
$C$ and $Y$.
\end{corollary}

We recall that $A\in\bal$ is \emph{Dedekind complete} if each subset of $A$ bounded above has a least upper bound, and hence each
subset bounded below has a greatest lower bound. Let $\dbal$ be the full subcategory of $\bal$ consisting of Dedekind complete objects
in $\bal$. By \cite[Thm.~3.3]{BMO13b}, $A\in\bal$ is Dedekind complete iff $A\in\ubal$ and $A$ is a Baer ring, where we recall that
a commutative ring is a \emph{Baer ring} provided each annihilator ideal is generated by an idempotent. Consequently,
$\dbal$ is a full subcategory of $\cubal$. By the Stone-Nakano theorem \cite{Sto40,Sto49,Nak41}, for $X\in\KHaus$ we have $C(X)$
is Dedekind complete iff $X\in\sf ED$. This together with Gelfand-Naimark-Stone duality yields:

\begin{corollary} \label{cor:dbal}
The categories $\sf ED$ and $\dbal$ are dually equivalent, and the dual equivalence is established by restricting the functors
$C$ and $Y$.
\end{corollary}

Let $\sf BA$ be the category of boolean algebras and boolean homomorphisms, and let $\sf CBA$ be the full subcategory of $\sf BA$
consisting of complete boolean algebras. By Stone duality for boolean algebras, $\sf BA$ is dually equivalent to $\sf Stone$ and
$\sf CBA$ is dually equivalent to $\sf ED$. Thus, we arrive at the following diagram:

\[
\begin{tikzcd}[column sep = 4pc, row sep = 1pc]
\bal \arrow[r, shift left = .5ex, "Y"]  & \creg \arrow[l, shift left = .5ex, "C^*"] & \\
\ubal \arrow[r, shift left = .5ex, "Y"] \arrow[u, hookrightarrow] & \KHaus \arrow[l, shift left = .5ex, "C"] \arrow[u, hookrightarrow] & \\
\cubal \arrow[r, leftrightarrow] \arrow[u, hookrightarrow] & \sf{Stone} \arrow[r, leftrightarrow] \arrow[u, hookrightarrow] & \sf{BA} \\
\dbal \arrow[u, hookrightarrow] \arrow[r, leftrightarrow] & \sf{ED} \arrow[u, hookrightarrow] \arrow[r, leftrightarrow] & \sf{cBA}
\arrow[u, hookrightarrow] \\
\end{tikzcd}
\]

Our goal is to generalize Gelfand-Naimark-Stone duality to compactifications and completely regular spaces. To achieve this, we require
new concepts of basic algebras and basic extensions.

\section{Basic algebras and a ring-theoretic version of Tarski duality}

Tarski duality establishes a dual equivalence between the category $\Set$ of sets and functions and the category $\CABA$ of
complete and atomic boolean algebras and complete boolean homomorphisms. The functor $\wp : \Set \to \CABA$ sends
a set $X$ to its powerset $\wp(X)$, and a function $\varphi : X \to Y$ to the boolean homomorphism $\varphi^{-1} : \wp(Y) \to \wp(X)$.
The functor $\At : \CABA \to \Set$ sends $B \in \sf CABA$ to the set $\At(B)$ of atoms of $B$.
If $\sigma : B \to C$ is a complete boolean homomorphism, then $\sf At$ sends $\sigma$ to $\sigma_* : \At(C) \to \At(B)$,
defined by $\sigma_*(c) = \bigwedge \{ b \in B \mid c \le \sigma(b) \}$.

We next use Corollary~\ref{cor:dbal} to give a ring-theoretic version of Tarski duality.
For $A\in\bal$, let $\Id(A)$ be the boolean algebra of idempotents of $A$. We recall
that an idempotent $e$ of $A$ is \emph{primitive} if $e \ne 0$ and $0 \leq f \leq e$ for some $f \in \Id(A)$ implies $f = 0$ or $f = e$.
Thus, primitive idempotents are exactly the atoms of the boolean algebra $\Id(A)$. Let ${\prim}(A)$ be the set of primitive idempotents of $A$.

Let $A\in\dbal$. Then $\Id(A)$ is complete. By Corollary~\ref{cor:dbal}, $Y_A\in\sf ED$ and $\alpha:A\to C(Y_A)$ is an isomorphism.
Moreover, $\Id(A)$ is isomorphic to the boolean algebra ${\sf Clop}(Y_A)$ of clopen subsets of $Y_A$. Primitive idempotents then
correspond to isolated points of $Y_A$.

\begin{convention}
Let $X_A$ be the set of isolated points of $Y_A$.
\end{convention}

As follows from the next lemma, the correspondence between primitive idempotents of $A$ and isolated points of $Y_A$ is obtained
by associating with each primitive idempotent $e$ the maximal $\ell$-ideal $(1-e)A$.

\begin{lemma} \label{prim id}
Let $A \in \dbal$ and $0 \ne e \in \Id(A)$. The following are equivalent.
\begin{enumerate}
\item $e$ is a primitive idempotent of $A$.
\item $(1-e)A$ is a maximal $\ell$-ideal of $A$.
\item For each $a \in A$, there is $r \in {\mathbb{R}}$ such that $ae = re$.
\end{enumerate}
Consequently, $X_A = \{(1-e)A \mid e \in \prim(A)\}$.
\end{lemma}

\begin{proof}
(1) $\Rightarrow$ (2). This is proved in \cite[Lem.~4.1]{BMO18c}.

(2) $\Rightarrow$ (3). Since $(1-e)A$ is a maximal $\ell$-ideal, $A/(1-e)A \simeq \mathbb{R}$. Therefore, $a + (1-e)A = r + (1-e)A$
for some $r \in \mathbb{R}$. So $a - r \in (1-e)A$. Since $1-e$ is an idempotent, this yields that $(a - r)(1-e) = (a - r)$.
Thus, $(a - r)e = 0$, and hence $ae = re$.

(3) $\Rightarrow$ (1). Let $f \in \Id(A)$ and $f \leq e$. There is $r \in {\mathbb{R}}$ such that $fe = re$. Since $f \leq e$ and
$e,f \in \Id(A)$,  we have $f = ef = re$. Therefore, $r = 0$ or $r = 1$. Consequently, $f = 0$ or $f = e$, proving
that $e$ is primitive.

To prove the last statement of the lemma, since $A\in\dbal$, we have $A \cong C(Y_A)$. Therefore, idempotents of $A$
correspond to characteristic functions of clopens of $Y_A$. So if $e\in\Id(A)$, then the sets $Z_\ell(e)$
and $Z_\ell(1-e)$ are complementary clopens. Now suppose $e$ is a primitive idempotent. By the equivalence
of (1) and (2), $(1-e)A$ is a maximal $\ell$-ideal, hence $Z_\ell(1-e)$ is a singleton, whose
complement is $Z_\ell(e)$. This yields that $(1-e)A$ is an isolated point of $Y_A$. Conversely, if $N$ is
an isolated point of $Y_A$, then $Y_A\setminus\{N\}$ is clopen. Since ${\sf Clop}(Y_A)$ is isomorphic to $\Id(A)$,
there is $e \in \Id(A)$ such that
$Y_A\setminus\{N\} = Z_\ell(e)$. But then $\{N\}= Z_\ell(1-e) = \{(1-e)A\}$, so
$e$ is a primitive idempotent by the equivalence of (1) and (2). Thus, $X_A = \{(1-e)A \mid e \in \prim(A)\}$.
\end{proof}

Let $A \in \bal$. It follows from the proof of \cite[Thm.~4.3]{BMO18c} that $A \cong B(X)$ for some set $X$ iff $A \in \dbal$
and $\Id(A)$ is atomic. For the reader's convenience we give a proof of this in Proposition~\ref{basic char}, along with another
equivalent condition that $A \cong B(X_A)$. For this we require the following definition.

\begin{definition}
For $A\in\bal$, define $\vartheta_A:A \rightarrow B(X_A)$ as the composition $\vartheta_A=\kappa_A\circ\zeta_A$
where
$\zeta_A : A \rightarrow C(Y_A)$ is the Yosida representation and $\kappa_A: C(Y_A) \rightarrow B(X_A)$ sends $f \in C(Y_A)$
to its restriction to $X_A$. Since both $\zeta_A$ and $\kappa_A$ are morphisms in $\bal$, so is $\vartheta_A$.
\end{definition}

For a set $X$ and $x \in X$ let
\[
N_x = \{ f \in B(X) \mid f(x) = 0\}.
\]
Note that if $X \in \creg$, then $M_x = N_x \cap C^*(X)$.

\begin{proposition} \label{basic char}
The following are equivalent for $A \in \bal$.
\begin{enumerate}
\item $A \in \dbal$ and $\Id(A)$ is atomic.
\item There is a set $X$ such that $A \cong B(X)$.
\item $\vartheta_A:A \rightarrow B(X_A)$ is an isomorphism.
\end{enumerate}
\end{proposition}

\begin{proof}
(3) $\Rightarrow$ (1). We have $B(X_A)\in\bal$ and infinite joins and meets of bounded subsets of $B(X_A)$ are pointwise,
hence exist in $B(X_A)$. Therefore, $B(X_A)$ is Dedekind complete, hence $B(X_A) \in \dbal$. In addition, idempotents of
$B(X_A)$ are exactly the characteristic functions of subsets of $X_A$, and primitive idempotents the characteristic functions
of singletons. Thus, the boolean algebra of idempotents of $B(X_A)$ is isomorphic to the powerset $\wp(X_A)$. It follows that
$\Id(B(X_A))$ is atomic.

(1) $\Rightarrow$ (2). Since $A \in \dbal$, $Y_A$ is extremally disconnected and $A \cong C(Y_A)$ by Corollary~\ref{cor:dbal}.
So since $\Id(A)$ is atomic, $X_A$ is dense in $Y_A$. Therefore, $Y_A$ is homeomorphic to the Stone-\v{C}ech compactification
$\beta(X_A)$ of the discrete space $X_A$ (see, e.g., \cite[p.~96]{GJ60}). Thus, $C(Y_A)\cong C(\beta(X_A))\cong C^*(X_A)$.
Since $X_A$ is discrete, $C^*(X_A)=B(X_A)$, yielding that $A \cong B(X_A)$.

(2) $\Rightarrow$ (3). We may assume that
$A = B(X)$. Then the primitive idempotents of $A$ are characteristic functions of points of $X$. We have
$N_x = (1-\chi_{\{x\}})A$, so $X_A = \{ N_x \mid x \in X\}$
by Lemma~\ref{prim id}. Let $f \in A$. Then $\vartheta_A(f)(N_x) = r$ iff $f-r \in N_x$ iff $f(x) = r$.
From this it follows that $\vartheta_A$ is 1-1. To see that $\vartheta_A$ is onto, since $A$ is Dedekind
complete, $\zeta_A : A \rightarrow C(Y_A)$ is an isomorphism. In addition, since $\Id(A)$ is atomic, as
we already pointed out, $Y_A$ is homeomorphic to the Stone-\v{C}ech
compactification of $X_A$. Therefore, $\kappa_A:C(Y_A) \rightarrow B(X_A)$ is onto. Thus, $\vartheta_A$
is onto, hence an isomorphism.
\end{proof}

This proposition motivates the following definition.

\begin{definition} \label{def: basic}
We call $A \in \bal$ a \emph{basic algebra} if $A$ is Dedekind complete and $\Id(A)$ is atomic.
\end{definition}

\begin{remark}
It is shown in \cite[Thm.~4.3]{BMO18c} that for $A\in\dbal$, the conditions of Proposition~\ref{basic char} are also equivalent to
$A$ having essential socle.
\end{remark}

We recall that a unital $\ell$-algebra homomorphism $\alpha:A\to B$ between $A,B\in\dbal$ is \emph{normal} if it preserves all
existing joins (and hence all existing meets). Let $\balg$ be the category of basic algebras and normal $\ell$-algebra homomorphisms.

\begin{convention}
For a map $\varphi:X\to Y$, define $\varphi^+ : B(Y)\to B(X)$ by $\varphi^+(f)=f\circ\varphi$.
\end{convention}

\begin{remark}
If $X,Y\in\creg$ and $\varphi:X\to Y$ is continuous, then $\varphi^* : C^*(Y)\to C^*(X)$ is the restriction of $\varphi^+:B(Y)\to B(X)$.
\end{remark}

It is easy to see that $\varphi^+$ is a normal $\ell$-algebra homomorphism. Thus, we have a contravariant functor $B:\Set\to\balg$
which associates with each set $X$ the basic algebra $B(X)$, and with each map $\varphi:X\to Y$ the normal $\ell$-algebra homomorphism
$\varphi^+ : B(Y)\to B(X)$.

The contravariant functor $X:\balg\to\Set$ is defined as follows. With each basic algebra $A$ we associate
the set $X_A$ of isolated points of $Y_A$. To define the action of the functor on morphisms, we recall that a
continuous map $\varphi:X\to Y$ is \emph{skeletal} if $F$ nowhere dense in $Y$ implies that $\varphi^{-1}(F)$
is nowhere dense in $X$, and that $\varphi$ is \emph{quasi-open} if $U$ nonempty open in $X$ implies that the
interior of $\varphi(U)$ is nonempty in $Y$. It is well known that the two concepts of skeletal and quasi-open
maps coincide in $\KHaus$.
Let $\alpha:A\to B$ be a normal $\ell$-algebra homomorphism. By \cite[Thm.~7.6]{BMO16}, $\alpha_*:Y_B\to Y_A$ is skeletal.
Therefore, it is quasi-open, and hence $\alpha_*$ sends isolated points of $Y_B$ to isolated points of $Y_A$.
Thus, the restriction of $\alpha_*$ to $X_B$ is a well-defined map from $X_B$ to $X_A$.

\begin{convention}
Let $\alpha_+ : X_B \to X_A$ be the restriction of $\alpha_* : Y_B \to Y_A$.
\end{convention}

Consequently, we have a contravariant functor $X:\balg\to\Set$ which associates with each $A\in\balg$ the set $X_A$ and with
each normal $\ell$-algebra homomorphism $\alpha:A\to B$ the map $\alpha_+:X_B\to X_A$.

\begin{theorem} \label{thm:tarski}
The categories $\sf Set$ and $\balg$ are dually equivalent, and the dual equivalence is established by the functors $B$ and $X$.
\end{theorem}

\begin{proof}
We define a natural transformation $\eta$ from the identity functor on ${\sf Set}$ to $XB$ by $\eta_X(x) = N_x$.
Given a function $\varphi: X \to X'$, we have the following diagram.
\[
\begin{tikzcd}[column sep=5pc]
X \arrow[r, "\eta_X"] \arrow[d, "\varphi"'] & X_{B(X)} \arrow[d, "(\varphi^+)_+"] \\
X' \arrow[r, "\eta_{X'}"'] & X_{B(X')}
\end{tikzcd}
\]
To see that the diagram is commutative, let $x \in X$. Then $\eta_{X'}(\varphi(x)) = N_{\varphi(x)}$. Also,
\begin{align*}
(\varphi^+)_+(\eta_X(x)) &= (\varphi^+)_+(N_x) = (\varphi^+)^{-1}(N_x) = \{ f \in B(X') \mid \varphi^+(f) \in N_x\} \\
&= \{ f \in B(X') \mid f(\varphi(x)) = 0\} = N_{\varphi(x)}.
\end{align*}
Thus, $\eta_{X'} \circ \varphi = (\varphi^+)_+ \circ \eta_X$, and hence
$\eta$ is a natural transformation. Since $\eta_X$ is a bijection,
we conclude that $\eta$ is a natural isomorphism.

We show that $\vartheta$ is a natural transformation from the identity functor on $\balg$ to $BX$.
Given a normal homomorphism $\alpha:A\to A'$, we have the following diagram.
\[
\begin{tikzcd}[column sep=5pc]
A \arrow[r, "\vartheta_A"] \arrow[d, "\alpha"'] & B(X_A) \arrow[d, "(\alpha_+)^+"] \\
A' \arrow[r, "\vartheta_{A'}"'] & B(X_{A'})
\end{tikzcd}
\]
The diagram factors into the larger diagram.
\[
\begin{tikzcd}[column sep=5pc]
A \arrow[r, "\zeta_A"] \arrow[d, "\alpha"'] \arrow[rr, bend left = 20, "\vartheta_A"] & C(Y_A) \arrow[d, "(\alpha_*)^*"]
\arrow[r, "\kappa_A"] & B(X_A) \arrow[d, "(\alpha_+)^+"] \\
A' \arrow[r, "\zeta_{A'}"'] \arrow[rr, bend right = 20, "\vartheta_{A'}"'] & C(Y_{A'}) \arrow[r, "\kappa_{A'}"'] & B(X_{A'})
\end{tikzcd}
\]
We have $(\alpha_*)^*\circ\zeta_A = \zeta_{A'} \circ \alpha$ by Gelfand-Naimark-Stone duality, so the left square commutes. To see that the right square commutes, let $f \in C(Y_A)$. Then
\[
(\alpha_+)^+(\kappa_A(f)) = (\alpha_+)^+(f|_{X_A}) = (f|_{X_A}) \circ \alpha_+.
\]
On the other hand,
\[
(\kappa_{A'} \circ (\alpha_*)^*)(f) = \kappa_{A'}(f\circ \alpha_*) = (f\circ \alpha_*)|_{X_{A'}}.
\]
For $M \in X_{A'}$, we have
\[
(f|_{X_A} \circ \alpha_+)(M) = f(\alpha_*(M)) = (f \circ \alpha_*)|_{X_A}(M),
\]
so $f|_{X_A} \circ \alpha_+ = (f \circ \alpha_*)|_{X_{A'}}$, and hence the right square commutes. Therefore, $(\alpha_+)^+ \circ \vartheta_A = \vartheta_{A'} \circ \alpha$, and so $\vartheta$ is a natural transformation. It is a natural isomorphism by
Proposition~\ref{basic char}. Thus, $B$ and $X$ yield a dual equivalence of $\sf Set$ and $\balg$.
\end{proof}

\begin{remark}
The duality of Theorem~\ref{thm:tarski} relates to Tarski duality as follows. Define a covariant functor $\Id : \balg \to \CABA$ by sending each $A \in \balg$ to the boolean algebra $\Id(A)$ and a normal homomorphism $\alpha : A \to B$ to its restriction to $\Id(A)$. It is easy to see that $\Id$ is well defined, and we have the following diagram.
\[
\begin{tikzcd}[column sep=5pc]
& \Set  \arrow[rd, "{\wp}", shift left=0.5ex] \arrow[ld, "B", shift left=0.5ex] & \\
\balg \arrow[rr, "{\Id}"'] \arrow[ru, "X", shift left=0.5ex] && \CABA \arrow[lu, "\At", shift left=0.5ex]
\end{tikzcd}
\]
The functor $\wp : \Set \to \CABA$ is the composition $\Id\circ B$,
and the composition $B\circ\At : \CABA \to \balg$ takes $C \in \CABA$ and sends it to $B(X)$, where $X$ is the set of isolated points
of the Stone space of $C$, so $B \circ \At \cong \Id$.
\end{remark}

\section{Compactifications and basic extensions} \label{sec: compactifications}

We recall (see, e.g., \cite[Sec.~3.5]{Eng89}) that a \emph{compactification} of a completely regular space $X$ is a pair $(Y,e)$,
where $Y$ is a compact Hausdorff space and $e : X\to Y$ is a topological embedding such that the image $e[X]$ is dense in $Y$.
Suppose that $e : X \to Y$ and $e' : X \to Y'$ are compactifications. As usual, we write $e \le e'$ provided there is a continuous map
$f : Y' \to Y$ with $f \circ e' = e$.
\[
\begin{tikzcd}
X \arrow[r, "e'"] \arrow[rd, "e"'] & Y' \arrow[d, "f"] \\
& Y
\end{tikzcd}
\]
The relation $\le$ is reflexive and transitive. Two compactifications $e$ and $e'$ are said to be \emph{equivalent} if $e \le e'$ and $e' \le e$.
It is well known that $e$ and $e'$ are equivalent iff there is a homeomorphism $f : Y'\to Y$ with $f \circ e' = e$. The equivalence classes of
compactifications of $X$ form a poset whose largest element is the Stone-\v{C}ech compactification $s : X \to \beta X$. There are many
constructions of $\beta X$.

\begin{convention} \label{con: beta}
We will follow Stone \cite{Sto37b} in viewing $\beta X$ as the maximal ideals of $C^*(X)$. Since maximal ideals of $C^*(X)$ are the same
as maximal $\ell$-ideals (see \cite[Lem.~1.1]{HIJ61} or \cite[Prop.~4.2]{BMO13a}), throughout this paper we identify $\beta X$ with
$Y_{C^*(X)}$, and hence the embedding $s : X \to \beta X$ sends $x$ to $M_x = \{ f\in C^*(X) \mid f(x) = 0\}$.
\end{convention}

In the classical setting, one considers compactifications of a fixed base space $X$. The following category of compactifications, without a fixed base space, was studied in \cite{BMO19a}.

\begin{definition}
Let $\C$ be the category whose objects are compactifications $e:X\to Y$ and whose morphisms are pairs $(f,g)$ of continuous maps such that
the following diagram commutes.
\[
\begin{tikzcd}[column sep=5pc]
X \arrow[r, "e"] \arrow[d, "f"'] & Y \arrow[d, "g"] \\
X' \arrow[r, "e'"'] & Y'
\end{tikzcd}
\]
The composition of two morphisms $(f_1,g_1)$ and $(f_2,g_2)$ is defined to be $(f_2\circ f_1, g_2\circ g_1)$.
\[
\begin{tikzcd}[column sep=5pc]
X_1 \arrow[r, "e_1"] \arrow[dd, bend right = 35, "f_2 \circ f_1"'] \arrow[d, "f_1"] & Y_1 \arrow[d, "g_1"']
\arrow[dd, bend left = 35, "g_2 \circ g_1"] \\
X_2 \arrow[r, "e_2"] \arrow[d, "f_2"] & Y_2 \arrow[d, "g_2"'] \\
X_3 \arrow[r, "e_3"'] & Y_3
\end{tikzcd}
\]
\end{definition}

It is straightforward to see that a morphism $(f,g)$ in $\C$ is an isomorphism iff both $f$ and $g$ are homeomorphisms.

\begin{convention}
For a compactification $e : X \to Y$ let $e^\flat : C(Y) \to B(X)$ be given by $e^\flat(f) = f\circ e$.
\end{convention}

\begin{remark}
We have that $e^\flat=\iota\circ e^*$ where $\iota:C^*(X)\to B(X)$ is the inclusion map.
\end{remark}

\begin{proposition} \label{functor 2}
If $e : X \to Y$ is a compactification,
then $e^\flat : C(Y) \to B(X)$ is a monomorphism in $\bal$, and each element of $B(X)$ is a join of meets of elements from $e^\flat[C(Y)]$.
\end{proposition}

\begin{proof}
That the $\ell$-algebra homomorphism $e^\flat : C(Y) \rightarrow B(X)$ is 1-1 follows from the fact that the image of $e$ is dense in $Y$.
Next we show that every primitive idempotent in $B(X)$ is a meet of elements from $e^\flat[C(Y)]$. Let $b \in {\prim}(B(X))$. Then $b$ is
the characteristic function of a singleton set, so $b = \chi_{\{x\}}$ for some $x \in X$. Let
\[
S = \{e^\flat(g) \mid 0 \leq g \in C(Y) {\mbox{ and }} g(e(x)) = 1\}.
\]
We claim $\chi_{\{x\}} = \bigwedge S$. It is clear that $\chi_{\{x\}} \leq \bigwedge S$. Suppose by way of contradiction that
$\chi_{\{x\}} \ne \bigwedge S$. Then there exist $x' \in X$ and $\varepsilon > 0$  such that $x \ne x'$ and $g(e(x')) > \varepsilon$
for all $0 \leq g \in C(Y)$ with $g(e(x)) = 1$.  Since $e$ is 1-1, we have $e(x) \ne e(x')$, and since $Y$ is completely regular,
there exists $0\le g \in C(Y)$ with $g(e(x)) =1 $ and $g(e(x'))  = 0$. This contradiction shows that $\chi_{\{x\}} = \bigwedge S$.
Therefore, every primitive idempotent in $B(X)$ is a meet of elements from $e^\flat[C(Y)]$.

Let $0 \leq c \in B(X)$. By Lemma~\ref{prim id}, for each $b \in {\prim}(B(X))$ there is $r_b \in {\mathbb{R}}$ such that $cb = r_b b$.
Since $\Id(B(X))$ is atomic, we have $1 = \bigvee \prim(B(X))$, and so since  $c \ge 0$,
\begin{align*}
c &= c \cdot 1 = c \cdot \bigvee \{ b \in \prim(B(X)) \} \\
&=\bigvee \{cb \mid b \in {\prim}(B(X))\} \\
&= \bigvee \{ r_b b \mid b \in {\prim}(B(X))\}.
\end{align*}
Since $c \ge 0$, each $r_b \ge 0$, and as the primitive idempotent $b$ is a meet of elements from $e^\flat[C(Y)]$, so is $r_b b$.
This yields that every positive element of $B(X)$ is a join of meets of elements from $e^\flat[C(Y)]$.

To finish the argument, let $c \in B(X)$. Then there is $s \in \mathbb{R}$ with $c+s \ge 0$. By the previous argument,
we may then write $c+s = \bigvee \{ r_b b \mid b \in {\prim}(B(X))\}$ for some $0 \le r_b \in \mathbb{R}$. Therefore,
\[
c = (c + s) - s = \bigvee \{ r_b b - s \mid b \in \prim(B(X))\}.
\]
Since $r_bb$ is a meet of elements from $e^\flat[C(Y)]$ and $s$ is a scalar, $r_bb-s$ is also a meet of elements from $e^\flat[C(Y)]$,
Thus, every element of $B(X)$ is a join of meets of elements from $e^\flat[C(Y)]$.
\end{proof}

\begin{definition}
Let $\alpha : A \to B$ be a monomorphism in $\bal$.
\begin{enumerate}
\item We say $\alpha[A]$ is join-meet dense in $B$ if each element of $B$ is a join of meets from $\alpha[A]$.
\item We say $\alpha[A]$ is meet-join dense in $B$ if each element of $B$ is a meet of joins from $\alpha[A]$.
\end{enumerate}
\end{definition}

\begin{remark} \label{rem: jm-dense}
Let $\alpha : A \to B$ be a monomorphism in $\bal$ with $B$ a basic algebra.
\begin{enumerate}
\item $\alpha[A]$ is join-meet dense in $B$ iff each primitive idempotent of $B$ is a meet from $\alpha[A]$. The right-to-left
implication follows from the proof of Proposition~\ref{functor 2}. For the left-to-right implication, let $b \in {\prim}(B)$.
Then $b$ is a join of meets of elements from $\alpha[A]$, so there is a meet $c$ of elements from $\alpha[A]$ with $0 < c$ and
$c \le b$. Because $0 \le 1-b$, we have $0 \le c(1-b) \le b(1-b) = 0$. Therefore, $c(1-b) = 0$, so $c = cb$. By Lemma~\ref{prim id},
$cb = rb$ for some nonzero scalar $r$. Thus, $r > 0$. This implies $b = r^{-1}c$ is a meet of elements from $\alpha[A]$.
\item $\alpha[A]$ is join-meet dense in $B$ iff $\alpha[A]$ is meet-join dense in $B$. For the left-to-right implication,
let $b \in B$. We may write $-b = \bigvee \{\bigwedge \{ \alpha(a_{ij}) \mid j \in J\} \mid i \in I\}$ for some $a_{ij} \in A$.
Then $b = \bigwedge \{ \bigvee \{ \alpha(-a_{ij}) \mid j \in J\} \mid i \in I\}$, which shows $b$ is a meet of joins from $\alpha[A]$.
The reverse implication is similar.
\end{enumerate}
\end{remark}

Proposition~\ref{functor 2} motivates the following key definition of the article.

\begin{definition}
Let $A\in\bal$, $B\in\balg$, and $\alpha : A \to B$ be a monomorphism in $\bal$. We call $\alpha : A \to B$ a \emph{basic extension}
if $\alpha[A]$ is join-meet dense in $B$.
\end{definition}

\begin{example} \label{ex: basic extensions}
\begin{enumerate}
\item[]
\item If $e : X \to Y$ is a compactification, then $e^\flat : C(Y) \to B(X)$ is a basic extension by Proposition~\ref{functor 2}.
\item If $X$ is completely regular, then the inclusion map $\iota: C^*(X) \to B(X)$ is a basic extension. To see this, let
$s : X \to \beta X$ be the Stone-\v{C}ech compactification of $X$. By (1), $s^\flat$ is a basic extension. Since $s^\flat$ is an
isomorphism from $C(\beta X)$ to $C^*(X)$, we see that $\iota$ is a basic extension. In fact, $s^\flat$ is isomorphic to $\iota$
in the category of basic extensions described in Definition~\ref{prop: deve a category}.
\item If $Y \in \KHaus$, it follows from (1) that the inclusion map $\iota : C(Y) \to B(Y)$ is a basic extension.
\item If $A \in \bal$, it follows from  \cite[Thm.~2.9]{BMO18c} that $\vartheta_A : A \to B(Y_A)$ is a basic extension.
\end{enumerate}
\end{example}

\begin{definition} \label{prop: deve a category}
\begin{enumerate}
\item[]
\item Let $\basic$ be the category whose objects are basic extensions and whose morphisms are pairs $(\rho, \sigma)$ of morphisms in $\bal$
with $\sigma$ normal and $\sigma\circ\alpha = \alpha' \circ \rho$.
\[
\begin{tikzcd}[column sep=5pc]
A \arrow[r, "\alpha"] \arrow[d, "\rho"'] & B \arrow[d, "\sigma"] \\
A' \arrow[r, "\alpha'"'] & B'
\end{tikzcd}
\]
The composition of two morphisms $(\rho_1, \sigma_1)$ and $(\rho_2, \sigma_2)$ is defined to be $(\rho_2\circ \rho_1, \sigma_2\circ \sigma_1)$.
\[
\begin{tikzcd}[column sep=5pc]
A_1 \arrow[dd, bend right = 35, "\rho_2 \circ \rho_1"'] \arrow[r, "\alpha_1"] \arrow[d, "\rho_1"] & B_1 \arrow[d, "\sigma_1"']
\arrow[dd, bend left=35, "\sigma_2 \circ \sigma_1"] \\
A_2 \arrow[r, "\alpha_2"] \arrow[d, "\rho_2"] & B_2 \arrow[d, "\sigma_2"'] \\
A_3 \arrow[r, "\alpha_3"'] & B_3
\end{tikzcd}
\]
\item Let $\ubasic$ be the full subcategory of $\basic$ consisting of the basic extensions $\alpha : A \to B$ where $A \in \ubal$.
\end{enumerate}
\end{definition}

\begin{theorem} \label{prop: uE reflective}
$\ubasic$ is a reflective subcategory of $\basic$.
\end{theorem}

\begin{proof}
It is well known (see, e.g., \cite[p.~447]{BMO13a}) that $\ubal$ is a reflective subcategory of $\bal$, and the reflector sends $A \in \bal$ to $C(Y_A)$.
If $\alpha : A \to C$ is a morphism in $\bal$ with $C \in \ubal$, let $\widehat{\alpha}:C(Y_A) \to C$ be the unique morphism
in $\bal$ with $\widehat{\alpha} \circ \zeta_A = \alpha$.

Define a functor $\sf r : \basic \to \ubasic$ as follows. If $\alpha : A \to B$ is a basic extension, then $\widehat{\alpha}[C(Y_A)]$ is
join-meet dense in $B$ because it contains $\alpha[A]$. Thus, $\widehat{\alpha} : C(Y_A) \to B$ is a basic extension. We set
${\sf r}(\alpha) = \widehat{\alpha}$. If $(\rho, \sigma)$ is a morphism in $\basic$, we set ${\sf r}(\rho, \sigma) = ((\rho_*)^*, \sigma)$.
To see that $((\rho_*)^*, \sigma)$ is a morphism in $\basic$, consider the following diagram.
\[
\begin{tikzcd}[column sep = 5pc]
A \arrow[rr, "\alpha"] \arrow[rd, "\zeta_A"'] \arrow[ddd, "\rho"'] && B \arrow[ddd, "\sigma"] \\
& C(Y_A) \arrow[d, "(\rho_*)^*"] \arrow[ru, "\widehat{\alpha}"'] & \\
& C(Y_{A'}) \arrow[rd, "\widehat{\alpha'}"] & \\
A' \arrow[rr, "\alpha'"'] \arrow[ru, "\zeta_{A'}"] && B'
\end{tikzcd}
\]
We have
\[
\sigma \circ \widehat{\alpha} \circ \zeta_A = \sigma \circ \alpha = \alpha' \circ \rho = \widehat{\alpha'} \circ \zeta_{A'} \circ \rho =
\widehat{\alpha'} \circ (\rho_*)^* \circ \zeta_A.
\]
Since $\zeta_A$ is epic (see, e.g., \cite[Lem.~2.9]{BMO13a}), $\sigma \circ \widehat{\alpha} = \widehat{\alpha'} \circ (\rho_*)^*$.
Therefore, $((\rho_*)^*, \sigma)$ is a morphism in $\basic$.
Since $C$ is a functor, it is straightforward to see that $\sf r$ is a functor.
For a basic extension $\alpha : A \to B$, we let ${\sf r}_\alpha$  be the morphism $(\zeta_A, 1_B)$ from $\alpha$ to $\widehat{\alpha}$,
where $1_B$ is the identity on $B$.
Suppose that $(\rho, \sigma)$ is a morphism to an object $\alpha' : A' \to B'$ of $\ubasic$, so $A' \in \ubal$.
\[
\begin{tikzcd}[column sep=5pc]
A \arrow[r, "\alpha"] \arrow[d, "\zeta_A"] \arrow[dd, bend right = 45, "\rho"'] & B \arrow[d, "1_B"'] \arrow[dd, bend left = 45, "\sigma"] \\
C(Y_A) \arrow[r, "\widehat{\alpha}"'] \arrow[d, "\widehat{\rho}"] & B \arrow[d, "\sigma"'] \\
A' \arrow[r, "\alpha'"'] & B'
\end{tikzcd}
\]
Then $(\widehat{\rho}, \sigma)$ is a morphism in $\ubasic$, and since $\widehat{\rho}$ is the unique morphism satisfying
$\widehat{\rho}\circ\zeta_A = \rho$, it follows that $(\widehat{\rho}, \sigma)$ is the unique morphism satisfying
$(\widehat{\rho}, \sigma) \circ (\zeta_A, 1_B) = (\rho, \sigma)$. This proves that $\ubasic$ is a reflective subcategory of $\basic$.
\end{proof}

Define a contravariant functor ${\sf E} : \C \to \basic$ as follows. If $e : X \to Y$ is a compactification, define
$\sf E(e)$ to be the basic extension $e^\flat : C(Y) \to B(X)$. For a morphism $(f,g)$ in $\C$
\[
\begin{tikzcd}[column sep = 5pc]
X \arrow[r, "e"] \arrow[d, "f"'] & Y \arrow[d, "g"] \\
X' \arrow[r, "e'"'] & Y'
\end{tikzcd}
\]
define ${\sf E}(f,g)$ to be the pair $(g^*, f^+)$
\[
\begin{tikzcd}[column sep = 5pc]
C(Y') \arrow[ r, "(e')^\flat"] \arrow[d, "g^*"'] & B(X') \arrow[d, "f^+"] \\
C(Y) \arrow[r, "e^\flat"'] & B(X),
\end{tikzcd}
\]

\begin{proposition} \label{prop: E}
${\sf E} : \C \to \basic$ is a contravariant functor such that each object of $\ubasic$ is isomorphic to
${\sf E}(e)$ for some compactification $e : X \to Y$.
\end{proposition}

\begin{proof}
Let $e : X \to Y$ be a compactification. By Proposition~\ref{functor 2}, $e^\flat : C(Y) \to B(X)$ is a basic extension.
Thus, ${\sf E}(e) \in \basic$. Let $(f,g)$ be a morphism in $\C$.
Then ${\sf E}(f,g) = (g^*, f^+)$. We show that ${\sf E}(f,g)$ is a morphism in $\basic$. Let $a \in C(Y')$. Then
\begin{align*}
(f^+\circ (e')^\flat)(a) &= f^+((e')^\flat(a)) = (a\circ e')\circ f = a\circ (e' \circ f) = a\circ (g \circ e) \\
&= (a \circ g)\circ e = e^\flat(g^*(a)) = (e^\flat\circ g^*)(a).
\end{align*}
This yields $f^+\circ (e')^\flat = e^\flat \circ g^*$. Because $g^*, f^+$ are morphisms in $\bal$ and $f^+$ is normal,
${\sf E}(f,g)$ is a morphism in $\basic$. From the definition of composition in $\C$ and $\basic$ it is elementary
to see that $\sf E$ preserves composition and identity morphisms. Thus, $\sf E$ is a contravariant functor.
That each object of $\ubasic$ is isomorphic to ${\sf E}(e)$ for some compactification $e : X \to Y$
follows from the definition of $\ubasic$ and Gelfand-Naimark-Stone duality.
\end{proof}

\section{The functor ${\sf C}: \basic \to \C$}

\begin{convention}
For a morphism $\alpha:A\to B$ in $\bal$, let $\alpha_\flat : X_B \to Y_A$ be the restriction of $\alpha_* : Y_B \to Y_A$ to $X_B$.
\end{convention}

\begin{definition} \label{def:tau_alpha}
Let $\alpha : A \to B$ be a monomorphism in $\bal$ with $B$ a basic algebra. Define a topology $\tau_\alpha$ on $X_B$ as the least topology
making $\alpha_\flat:X_B \to Y_A$ continuous.
\end{definition}

\begin{remark}
Following usual terminology, we will refer to the topological space $(X_B, \tau_\alpha)$ as $X_B$ when there is no danger of confusion
about which topology we are using.
\end{remark}

\begin{proposition} \label{jm char}
Let $\alpha : A \to B$ be a monomorphism in $\bal$ with $B$ a basic algebra. Then $\alpha_\flat : X_B \to Y_A$ is \emph{1-1} iff
$\alpha[A]$ is join-meet dense in $B$.
\end{proposition}

\begin{proof}
By Remark~\ref{rem: jm-dense}(1), it is sufficient to show that $\alpha_\flat$ is 1-1 iff each $b \in {\prim}(B)$ is meet of elements
from $\alpha[A]$.

First suppose that $\alpha_\flat$ is 1-1. Let $b \in {\func{Prim}}(B)$. We show that $1-b$ is a join of elements from $\alpha[A]$.
Let $g$ be the join of all $\alpha(a) \in A$ with $0 \le \alpha(a) \le 1-b$. Clearly $g \le 1-b$. To see that $1-b\le g$, let
$c \in {\func{Prim}}(B)$ with $c \ne b$. Since $\alpha_\flat$ is 1-1, there is $\alpha(a) \in (1-b)B$ with $\alpha(a) \notin (1-c)B$.
By replacing $a$ by $|a|$, and then multiplying by an appropriate scalar, we may assume $0 \le a \le 1$. Since $\alpha(a) \in (1-b)B$
and $1-b$ is an idempotent,
$\alpha(a) = (1-b)\alpha(a)$, so $\alpha(a)b = 0$. Also $\alpha(a) \not \in (1-c)B$ implies $\alpha(a)(1-c) \ne \alpha(a)$, so $\alpha(a)c\ne 0$.
Since $a > 0$, Lemma~\ref{prim id} yields $\alpha(a)c = rc$ for some real number $r > 0$.
We have $c = r^{-1}rc = r^{-1}\alpha(a)c \leq r^{-1}\alpha(a)$ since $c \leq 1$.
Let $a' = 1 \wedge r^{-1}a$. Then $a' \in A$ and $c \le \alpha(a')$. Moreover, $0 \le a' \le r^{-1}a$, so
$0 \le \alpha(a')b \le r^{-1}\alpha(a)b = 0$. This implies $\alpha(a')b = 0$, so $\alpha(a')(1-b) = \alpha(a')$.
Since $\alpha(a')\le 1$, we have $\alpha(a')(1-b)\le 1-b$, hence $\alpha(a')\le 1-b$. Therefore,
$\alpha(a') \le g$ by the definition of $g$. This yields $c \le \alpha(a') \le g$ for each $c \ne b$.
Since $1-b$ is the join of all primitive idempotents $c \ne b$, we get $1-b \le g \le 1-b$, so $1-b = g$.
This shows that $1-b$ is a join of elements from $\alpha[A]$, and hence $b$ is a meet of elements from $\alpha[A]$.

Conversely, suppose that each $b \in {\prim}(B)$ is a meet of elements from $\alpha[A]$. Let $M,N \in X_B$ and $\alpha^{-1}(M) = \alpha^{-1}(N)$.
By Lemma~\ref{prim id}, there are $b,c \in {\func{Prim}}(B)$ with $M = (1-b)B $ and $N = (1-c)B$. Thus, $\alpha^{-1}((1-b)B) = \alpha^{-1}((1-c)B)$.
As $b$ is primitive, $1-b$ is the join of the elements from $\alpha[A]$ below it. Because $0 \le 1-b$, these elements can
be taken to be positive. Take $a \in A$ with $0 \le \alpha(a) \le 1-b$. By the calculation in the first paragraph above,
$\alpha(a)b = 0$, and so $\alpha(a) = \alpha(a)(1-b)$. Therefore, $a \in \alpha^{-1}((1-b)B) = \alpha^{-1}((1-c)B)$, and
hence $\alpha(a)\in (1-c)B$. Since $1-c$ is an idempotent, it follows that
$\alpha(a) = \alpha(a)(1-c)$, which implies $\alpha(a) \le 1-c$. As $1-b$ is the join of all such $\alpha(a)$ we see that $1-b \le 1-c$.
Because $b, c$ are primitive idempotents, this implies $b = c$, and so $M = N$. Thus, $\alpha_\flat$ is 1-1.
\end{proof}

\begin{theorem} \label{lem:compactification}
If $\alpha : A \to B$ is a basic extension, then $X_B$ is completely regular and $\alpha_\flat : X_B \to Y_A$ is a compactification.
\end{theorem}

\begin{proof}
By Proposition~\ref{jm char}, $\alpha_\flat : X_B \to Y_A$ is 1-1. Therefore, $X_B$ is homeomorphic to $\alpha_\flat[X_B]$ by
Definition~\ref{def:tau_alpha}, and so $X_B$ is completely regular. Since $\alpha$ is 1-1, $\alpha_* : Y_B \to Y_A$ is onto,
and $X_{B}$ is dense in $Y_B$ because $B$ is atomic. Thus, $\alpha_\flat[X_B]$ is dense in $Y_A$, and hence $\alpha_\flat : X_{B} \to Y_A$
is a compactification.
\end{proof}

\begin{example} \label{ex: 5.5}
Let $X$ be a completely regular space. By Example~\ref{ex: basic extensions}(2), the inclusion $\iota : C^*(X) \to B(X)$ is a basic
extension. Then $\iota_\flat : X_{B(X)} \to Y_{C^*(X)}$ is a compactification by Theorem~\ref{lem:compactification}. We claim that
$\iota_\flat$ is isomorphic to the Stone-\v{C}ech compactification $s : X \to \beta X$ in $\C$. By Convention~\ref{con: beta},
$\beta X = Y_{C^*(X)}$ and $s(x) = M_x$. Consider the following diagram
\[
\begin{tikzcd}
X \arrow[r, "s"] \arrow[d, "\eta_X"'] & Y_{C^*(X)} \arrow[d, equal] \\
X_{B(X)} \arrow[r, "\iota_\flat"'] & Y_{C^*(X)}
\end{tikzcd}
\]
where we recall from the proof of Theorem~\ref{thm:tarski} that $\eta_X : X \to X_{B(X)}$, sending $x$ to $N_x$, is a bijection.
The diagram commutes because
\[
\iota_\flat(\eta_X(x)) = \iota^{-1}(N_x) = \{ f \in C^*(X) \mid f(x) = 0\} = M_x = s(x).
\]
To see that $\eta_X$ is a homeomorphism, since $(X_{B(X)}, \tau_\iota)$ is completely regular, it has a basis of cozero sets.
Let $U$ be a cozero set in $X_{B(X)}$. Then there is $f \in C^*(X)$ with $U = \{ M \in X_{B(X)} \mid f \notin M\}$. We have
\[
\eta_X^{-1}(U) = \{ x \in X \mid N_x \in U \} = \{ x\in X \mid f \notin N_x\} = \{ x \in X \mid f(x) \ne 0 \},
\]
which is the cozero set of $f$ in $X$. Since $\eta_X$ is a bijection, we conclude that $\eta_X$ is a homeomorphism.
Thus, $\iota_\flat$ and $s$ are isomorphic in $\C$.
\end{example}

\begin{lemma} \label{lem: commute}
Suppose that $e : X \to Y$ and $e' : X' \to Y'$ are compactifications, $g : Y \to Y'$ is a continuous map, and $f : X \to X'$ is
a map such that $e'\circ f = g \circ e$. Then $f$ is continuous.
\[
\begin{tikzcd}[column sep = 5pc]
X \arrow[r, "e"] \arrow[d, "f"'] & Y \arrow[d, "g"] \\
X' \arrow[r, "e'"'] & Y'
\end{tikzcd}
\]
\end{lemma}

\begin{proof}
Let $U$ be open in $X'$. Then there is an open set $V$ of $Y'$ with $U = (e')^{-1}(V)$. We have
\[
f^{-1}(U) = f^{-1}((e')^{-1}(V)) = (e'\circ f)^{-1}(V) = (g\circ e)^{-1}(V) = e^{-1}(g^{-1}(V)),
\]
so $f^{-1}(U)$ is open in $X$. Thus, $f$ is continuous.
\end{proof}

Define a functor ${\sf C} : \basic \to \C$ as follows. If $\alpha : A \to B$ is a basic extension, set ${\sf C}(\alpha)$ to be
the compactification $\alpha_\flat : X_B \to Y_A$. For a morphism $(\rho, \sigma)$ in $\basic$
\[
\begin{tikzcd}[column sep = 5pc]
A \arrow[r, "\alpha"] \arrow[d, "\rho"'] & B \arrow[d, "\sigma"] \\
A' \arrow[r, "\alpha'"'] & B'
\end{tikzcd}
\]
define ${\sf C}(\rho, \sigma)$ to be $(\sigma_+, \rho_*)$, where $\sigma_+$ is the restriction of $\sigma_* : Y_{B'} \to Y_B$ to $X_{B'}$.
\[
\begin{tikzcd}[column sep = 5pc]
X_{B'} \arrow[r, "\alpha'_\flat"] \arrow[d, "\sigma_+"'] & Y_{A'} \arrow[d, "\rho_*"] \\
X_B \arrow[r, "\alpha_\flat"'] & Y_A
\end{tikzcd}
\]
Since $\sigma\circ \alpha = \alpha' \circ \rho$ we have $\rho_* \circ \alpha'_* = \alpha_* \circ \sigma_*$. Restricting both sides to $X_{B'}$
shows this diagram is commutative. As an immediate consequence of Lemma~\ref{lem: commute}, we have:

\begin{lemma} \label{lem: sigma plus}
$\sigma_+$ is continuous, and hence $(\sigma_+, \rho_*)$ is a morphism in $\C$.
\end{lemma}

\begin{proposition} \label{prop: C}
${\sf C} : \basic \to \C$ is a contravariant functor.
\end{proposition}

\begin{proof}
Let $\alpha : A \to B$ be a basic extension. By Theorem~\ref{lem:compactification}, $\alpha_\flat : X_B \to Y_A$ is a compactification.
Therefore, ${\sf C}(\alpha) \in \C$. By Lemma~\ref{lem: sigma plus}, if $(\rho, \sigma)$ is a morphism in $\basic$, then
${\sf C}(\rho, \sigma) = (\sigma_+, \rho_*)$ is a morphism in $\C$. Suppose that $(\rho_1, \sigma_1)$ and $(\rho_2, \sigma_2)$
are composable morphisms in $\basic$. Then
\[
{\sf C}((\rho_2, \sigma_2)\circ (\rho_1, \sigma_1)) = {\sf C}(\rho_2\circ \rho_1, \sigma_2 \circ \sigma_1) =
((\sigma_2\circ \sigma_1)_+, (\rho_2\circ \rho_1)_*).
\]
Since $(\rho_2\circ\rho_1)_* = (\rho_1)_* \circ (\rho_2)_*$ and $(\sigma_2\circ\sigma_1)_+ = (\sigma_1)_+ \circ (\sigma_2)_+$, we see that
\begin{align*}
{\sf C}((\rho_2, \sigma_2)\circ (\rho_1, \sigma_1)) &= ((\sigma_1)_+ \circ (\sigma_2)_+, (\rho_1)_* \circ (\rho_2)_*) \\
&= ((\sigma_1)_+ , (\rho_1)_*) \circ ((\sigma_2)_+, (\rho_2)_*) = {\sf C}(\rho_1, \sigma_1)\circ {\sf C}(\rho_2, \sigma_2),
\end{align*}
which shows that $\sf C$ preserves composition. It is clear that $\sf C$ preserves identity morphisms.
Thus, $\sf C$ is a contravariant functor.
\end{proof}

\section{Duality between $\C$ and $\ubasic$}

In this section we prove that the functors $\sf E$ and $\sf C$ yield a dual adjunction between $\C$ and $\basic$, which restricts to
a dual equivalence between $\C$ and $\ubasic$. For this we require the following two lemmas.

\begin{lemma} \label{EC}
Let $\alpha : A \to B$ be a basic extension. Then $(\zeta_A, \vartheta_B)$ is a morphism in $\basic$, and it is an isomorphism
provided $A \in \ubal$.
 \[
 \begin{tikzcd}[column sep = 5pc]
 A \arrow[r, "\alpha"] \arrow[d, "\zeta_A"'] & B \arrow[d, "\vartheta_B"] \\
 C(Y_A) \arrow[r, "(\alpha_\flat)^\flat"'] & B(X_B)
 \end{tikzcd}
 \]
\end{lemma}

\begin{proof}
The map $\vartheta_B$ is an isomorphism by Proposition~\ref{basic char}, and so it is a normal homomorphism. To see that
$(\zeta_A, \vartheta_B)$ is a morphism in $\basic$ we need to show $(\alpha_\flat)^\flat \circ\zeta_A = \vartheta_B \circ \alpha$.
We have the following diagram.
 \[
 \begin{tikzcd}[column sep=5pc]
 A \arrow[r, "\alpha"] \arrow[d, "\zeta_A"'] & B \arrow[d, "\zeta_B"]  \arrow[dr, "\vartheta_B"] \\
 C(Y_A) \arrow[r, "(\alpha_*)^*"] \arrow[rr, bend right = 15, "(\alpha_\flat)^\flat"'] & C(Y_B) \arrow[r, "\kappa_B"] & B(X_B)
 \end{tikzcd}
 \]
The left-hand side commutes by Gelfand-Naimark-Stone duality, and the right-hand side commutes by the definition of $\vartheta_B$.
Therefore, $(\alpha_\flat)^\flat \circ\zeta_A = \vartheta_B \circ \alpha$, and hence $(\zeta_A, \vartheta_B)$ is a morphism in $\basic$.
If $A \in \ubal$, then $\zeta_A$ is an isomorphism. Since $\vartheta_B$ is an isomorphism, $(\zeta_A, \vartheta_B)$ is an isomorphism.
\end{proof}

\begin{lemma} \label{CE}
Let $e : X \rightarrow Y$ be a compactification. Then $\eta_X: X \to X_{B(X)}$
is a homeomorphism, and $(\eta_X, \xi_Y)$ is an isomorphism of compactifications between $e$ and $(e^\flat)_\flat$.
\[
\begin{tikzcd}[column sep = 5pc]
X \arrow[r, "e"] \arrow[d, "\eta_X"'] & Y \arrow[d, "\xi_Y"] \\
X_{B(X)} \arrow[r, "(e^\flat)_\flat"'] & Y_{C(Y)}
\end{tikzcd}
\]
\end{lemma}

\begin{proof}
We already observed in the proof of Theorem~\ref{thm:tarski} that $\eta_X$ is a bijection.
We show that $\xi_Y\circ e = (e^\flat)_\flat \circ \eta_X$. Let $x \in X$. Then $\xi_Y(e(x)) = M_{e(x)}$. We have
\[
(e^\flat)_\flat(\eta_X(x)) = \{ c \in C(Y) \mid e^\flat(c)(x) = 0 \} = \{ c \in C(Y) \mid c(e(x)) = 0\} = M_{e(x)}.
\]
Therefore, $(e^\flat)_\flat \circ \eta_X = \xi_Y \circ e$. Then, by Lemma~\ref{lem: commute}, $\eta_X$ is continuous, and so $(\eta_X, \xi_Y)$
is a morphism in $\C$. Since $\eta_X$ is a bijection and $\xi_Y$ is a homeomorphism, applying Lemma~\ref{lem: commute} (with the pair
$\eta_X^{-1}$ and $\xi_Y^{-1}$), shows that $\eta_X$  is a homeomorphism. Thus, $(\eta_X, \xi_Y)$ is an isomorphism in $\C$.
\end{proof}

\begin{theorem} \label{duality cptf}
The functors $\sf E:\C \rightarrow \basic$ and $\sf C:\basic \rightarrow \C$ define a dual adjunction of categories that restricts to a
dual equivalence between $\C$ and $\ubasic$.
\end{theorem}

\begin{proof}
Propositions~\ref{prop: E} and \ref{prop: C} show that $\sf E$ and $\sf C$ are contravariant functors. We first show that $\sf CE$
is naturally isomorphic to the identity functor on $\C$. The functor $\sf E$ sends $e : X \to Y$ to $e^\flat : C(Y) \to B(X)$. Then $\sf C$
sends this to $(e^\flat)_\flat : X_{B(X)} \to Y_{C(Y)}$. By Lemma~\ref{CE}, $(\eta_X, \xi_Y)$ is an isomorphism in $\C$.
\[
\begin{tikzcd}[column sep=5pc]
X \arrow[r, "e"] \arrow[d, "\eta_X"'] & Y \arrow[d, "\xi_Y"] \\
X_{B(X)} \arrow[r, "(e^\flat)_\flat"'] & Y_{C(Y)}
\end{tikzcd}
\]

Let $(f,g)$ be be a morphism in $\C$.
Then ${\sf E}(f,g) = (g^*,  f^+)$, and so ${\sf CE}(f,g) = {\sf C}(g^*, f^+) = ((f^+)_+, (g^*)_*)$.
Thus, $((f^+)_+, (g^*)_*)$ is a morphism in $\C$.
\[
\begin{tikzcd}[column sep = 5pc]
X_{B(X)} \arrow[r, "(e^\flat)_\flat"] \arrow[d, "(f^+)_+"'] & Y_{C(Y)} \arrow[d, "(g^*)_*"] \\
X_{B(X')} \arrow[r, "((e')^\flat)_\flat"'] & Y_{C(Y')}
\end{tikzcd}
\]
We define a natural transformation $p$ from the identity functor on $\C$ to $\sf CE$ as follows. For a compactification $e : X \to Y$
we set $p_e = (\eta_X, \xi_Y)$. By Lemma~\ref{CE}, $p_e$ is an isomorphism in $\C$. We have the following diagram.
\[
\begin{tikzcd}[column sep = 5pc]
& X_{B(X)} \arrow[rr, "(e^\flat)_\flat"] \arrow[dd, near start ,"(f^+)_+"] && Y_{C(Y)} \arrow[dd,"(g^*)_*"] \\
X \ar[ru, "\eta_X"] \arrow[dd, "f"'] \arrow[rr, near end, "e"] && Y \arrow[dd, near start, "g"'] \arrow[ru, "\xi_Y"] & \\
& X_{B(X')} \arrow[rr, near start, "((e')^\flat)_\flat"'] && Y_{C(Y')} \\
X' \arrow[rr,  "e'"'] \arrow[ru, "\eta_{X'}"] && Y' \arrow[ru, "\xi_{Y'}"] &
\end{tikzcd}
\]
The front and back faces of this cube are commutative because $(f,g)$ and ${\sf CE}(f,g) = ((f^+)_+, (g^*)_*)$ are morphisms in $\C$.
The top and bottom faces are commutative since $p_e$ and $p_{e'}$ are morphisms in $\C$. The right face is commutative by
Gelfand-Naimark-Stone duality, and the left face is commutative by Theorem~\ref{thm:tarski}.
The commutativity of this cube shows that $p$ is a natural transformation. In fact, since $p_e$ is an isomorphism in $\C$,
we see that $p$ is a natural isomorphism.

We next define a natural transformation $q$ from the identity functor on $\basic$ to $\sf EC$. Given a basic extension $\alpha : A \to B$,
the functor $\sf C$ sends it to $\alpha_\flat : X_B \to Y_A$. This is then sent by $\sf E$ to $(\alpha_\flat)^\flat : C(Y_A) \to B(X_B)$.
\[
\begin{tikzcd}[column sep=5pc]
A \arrow[r, "\alpha"] \arrow[d, "\zeta_A"'] & B \arrow[d, "\vartheta_B"] \\
C(Y_A) \arrow[r, "(\alpha_\flat)^\flat"'] & B(X_B)
\end{tikzcd}
\]
The pair $(\zeta_A, \vartheta_B)$ is a morphism in $\basic$ by Lemma~\ref{EC}.
Define $q$ by setting $q_\alpha = (\zeta_A, \vartheta_B)$
for a basic extension $\alpha : A \to B$.
By Lemma~\ref{EC}, $q_\alpha$ is a morphism in $\basic$. To show naturality, let $(\rho, \sigma)$ be a morphism in $\basic$.
We have the following diagram.
\[
\begin{tikzcd}[column sep = 5pc]
& C(Y_A) \arrow[rr, "(\alpha_\flat)^\flat"] \arrow[dd, near start, "(\rho_*)^*"] && B(X_B) \arrow[dd, "(\sigma_+)^+"] \\
A \arrow[ru, "\zeta_A"] \arrow[dd, "\rho"'] \arrow[rr, near end, "\alpha"] && B \arrow[dd, near start, "\sigma"'] \arrow[ru, "\vartheta_B"] & \\
& C(Y_{A'}) \arrow[rr, near start, "((\alpha')_\flat)^\flat"'] && B(X_{B'}) \\
A' \arrow[rr, "\alpha'"'] \arrow[ru, "\zeta_{A'}"] && B' \arrow[ru, "\vartheta_{B'}"] &
\end{tikzcd}
\]
The front and back faces of this cube are commutative because $(\rho, \sigma)$ and ${\sf EC}(\rho, \sigma) = ((\rho_*)^*, (\sigma_+)^+)$
are morphisms in $\basic$. The top and bottom faces are commutative because $q_\alpha$ and $q_{\alpha'}$ are morphisms in $\basic$.
The left face is commutative by Gelfand-Naimark-Stone duality, and the right face is commutative by Theorem~\ref{thm:tarski}.
This shows that $q$ is a natural transformation. In addition, if $\alpha : A \to B$ is an object of $\ubasic$, then $q_\alpha$ is an
isomorphism by Lemma~\ref{EC}. Therefore, $\sf C$ and $\sf E$ yield a dual equivalence between $\C$ and $\ubasic$. This together with Theorem~\ref{prop: uE reflective} gives that $\sf E$ and $\sf C$ define a dual adjunction between $\C$ and $\basic$.
\end{proof}

We conclude this section by relating Theorem~\ref{duality cptf} to Gelfand-Naimark-Stone duality.

\begin{remark}
Let ${\sf I} : \KHaus \to \C$ be the functor sending a compact Hausdorff space $Y$ to the compactification $1_Y : Y \to Y$ where $1_Y$
is the identity map. Let also ${\sf G} : \bal \to \basic$ be the functor sending $A\in\bal$ to the basic extension $\vartheta_A:A \to B(Y_A)$
(see Example~\ref{ex: basic extensions}(3)). We have the following diagram.
\[
\begin{tikzcd}[column sep=5pc]
\bal \ar[r, shift left = .5ex, "Y"] \arrow[d, "\sf G"'] & \KHaus \ar[l, shift left = .5ex, "C"] \arrow[d, "\sf I"] \\
\basic \ar[r, shift left = .5ex, "\sf C"] & \C \arrow[l, shift left = .5ex, "\sf E"]
\end{tikzcd}
\]
Let $A \in \bal$. Then ${\sf I}(Y(A))$ is the compactification $Y_A \to Y_A$. On the other hand, ${\sf C}({\sf G}(A))$ is the
image under $\sf C$ of the basic extension $A \to B(Y_A)$, which is $X_{B(Y_A)} \to Y_A$, and this is naturally isomorphic to $Y_A \to Y_A$.
Therefore, ${\sf I}\circ Y$ and ${\sf C}\circ {\sf G}$ are naturally isomorphic.

Next, let $Y \in \KHaus$. Then ${\sf G}(C(Y))$ is the basic extension $C(Y) \to B(Y_{C(Y)})$. Also, ${\sf E}({\sf I}(Y))$ is the image under
$\sf E$ of the compactification $Y \to Y$, which is the extension $C(Y) \to B(Y)$. Since $Y$ and $Y_{C(Y)}$ are naturally homeomorphic,
${\sf G}\circ C$ and ${\sf E}\circ {\sf I}$ are naturally isomorphic. Consequently, the duality of Theorem~\ref{duality cptf} extends
Gelfand-Naimark-Stone duality.
\end{remark}

\section{Duality for completely regular spaces}

In this final section we show how to use Theorem~\ref{duality cptf} to derive duality for the category $\creg$ of completely regular spaces.
For this we will need to introduce the concept of a maximal basic extension and connect it to the Stone-\v{C}ech compactification.

Since the Stone-\v{C}ech compactification is the largest among compactifications of a given completely regular space, it is natural to define
a maximal basic extension as the largest with respect to the corresponding order among basic extensions $\alpha:A\to B$ that yield the same
completely regular topology on $X_B$. We call such basic extensions compatible. To give a purely algebraic description of compatibility
requires some preparation.

Let $e : X \to Y$ and $e' : X \to Y'$ be two compactifications of the same completely regular space $X$. Then we have two basic extensions
$e^\flat : C(Y) \to B(X)$ and $(e')^\flat : C(Y') \to B(X)$. While the images of $C(Y)$ and $C(Y')$ in $B(X)$ are in general different,
as we will see shortly,
they have isomorphic Dedekind completions.
For this we need to recall Dilworth's notion of a
normal lower semicontinuous function \cite{Dil50}.

Let $X$ be completely regular.
For $x \in X$ let $\mathcal{N}_x$ be the family of open neighborhoods of $x$. For $f \in B(X)$ set
\[
f^*(x) = \inf_{U \in \mathcal{N}_x} \sup_{y \in U} f(y) \quad \mbox{and} \quad f_*(x) = \sup_{U \in \mathcal{N}_x} \inf_{y \in U} f(y).
\]
We recall that $f$ is \emph{lower semicontinuous} if $f=f_*$, \emph{upper semicontinuous} if $f=f^*$, and \emph{normal $($lower semicontinuous$)$}
if $(f^*)_* = f$. We set $N(X) = \{ f\in B(X) \mid (f^*)_* = f\}$.

\begin{remark} \label{rem: D(A)}
\begin{enumerate}
\item[]
\item Dilworth \cite{Dil50} showed that if we
view $C^*(X)$ and $N(X)$ as lattices, then $N(X)$ is the Dedekind completion of $C^*(X)$.
\item D\u{a}ne\c{t} \cite{Dan15} showed that Dilworth's result generalizes to the setting of vector lattices, and hence $N(X)$ is the
Dedekind completion of $C^*(X)$ as a vector lattice.
\item Clearly $C^*(X) \in \bal$. It follows from \cite[Ex.~8.4(2)]{BMO16} that there is a multiplication on $N(X)$ extending the multiplication
on $C^*(X)$ such that $N(X) \in \bal$. Consequently, $N(X)$ is the Dedekind completion of $C^*(X)$ also as an $\ell$-algebra.\footnote{We point out that neither the lattice  nor the algebra operations on $N(X)$ are pointwise.}
\end{enumerate}
\end{remark}

\begin{lemma} \label{rem: lifting}
Let $X$ be a subspace of a topological space $Y$.  If $f$ is a bounded lower $($resp.~upper$)$ semicontinuous real-valued
function on $X$, then there is a bounded lower $($resp.~upper$)$ semicontinuous  real-valued function $g$ on $Y$ with $g|_X = f$.
\end{lemma}

\begin{proof}
Let $f$ be a bounded lower semicontinuous function on $X$. Then $s := \sup\{ f(x) \mid x \in X\}$ exists. We extend $f$ to a function $f'$
on $Y$ by setting $f'(y) = s$ for all $y \in Y \setminus X$. Then $f'$ is a bounded function on $Y$. We define $g$ on $Y$ by setting
\[
g(y) = \sup_{U \in \mathcal{N}_y} \inf_{z \in U} f'(z)
\]
for each $y \in Y$. Then $g$ is lower semicontinuous by \cite[Sec.~3]{Dil50}. Let $x \in X$. By the definition of $f'$, if
$U \in \mathcal{N}_x$, then $\inf \{ f'(y) \mid y \in U\} = \inf \{ f(z) \mid z \in U \cap X\}$. Therefore,
\[
g(x) = \sup_{U \in \mathcal{N}_x} \inf_{y \in U} f'(y) = \sup_{U \in \mathcal{N}_x} \inf_{z \in U \cap X} f(z).
\]
Because $X$ is a subspace of $Y$, we see that $\{ U \cap X \mid U \in \mathcal{N}_x \}$ is the collection of open neighborhoods of
$x$ in $X$. Thus, since $f$ is lower semicontinuous on $X$,
\[
f(x) = \sup_{U \in \mathcal{N}_x} \inf_{z \in U \cap X} f(z) = g(x)
\]
for each $x \in X$, and hence $g|_X = f$. The argument for upper semicontinuous functions is similar and left to the reader.
\end{proof}

\begin{lemma} \label{lem: pointwise join-meet}
Let $\alpha : A \to B$ be a basic extension and let $f \in C^*(X_B)$. Then $f$ is a pointwise join and meet of elements from
$\alpha_\flat^*\zeta_A[A]$.
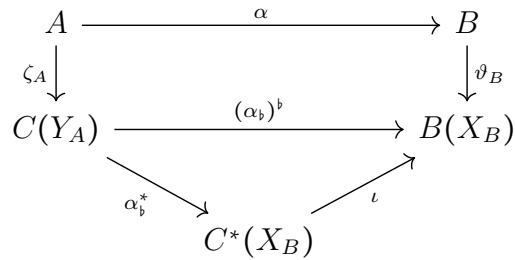
\[
\begin{tikzcd}
A \arrow[rr, "\alpha"] \arrow[d, "\zeta_A"'] && B \arrow[d, "\vartheta_B"] \\
C(Y_A) \arrow[rr, "(\alpha_\flat)^\flat"] \arrow[rd, "\alpha_\flat^*"'] && B(X_B) \\
& C^*(X_B) \arrow[ru, "\iota"'] &
\end{tikzcd}
\]
\end{lemma}

\begin{proof}
Since $f$ is continuous, by Lemma~\ref{rem: lifting} there is a lower semicontinuous function $g \in B(Y_A)$ and an upper semicontinuous function
$h \in B(Y_A)$ with $g \circ \alpha_\flat = f = h \circ \alpha_\flat$.
By \cite[Lem.~4.1]{Dil50} (and its dual for lower semicontinuous functions), $g = \bigvee S$ is a pointwise join of elements from $C(Y_A)$
and $h$ is a pointwise
meet of elements from $C(Y_A)$. Because the map $(\alpha_\flat)^+ : B(Y_A) \to B(X_B)$ sending $f$ to $f\circ \alpha_\flat$ is a normal
homomorphism,
\[
f = g \circ \alpha_\flat = \alpha_\flat^+(g)  = \bigvee \{ \alpha_\flat^+(k)  \mid k \in S\} = \bigvee \{ \alpha_\flat^*(k)  \mid k \in S\}
\]
is a join of elements from $(\alpha_\flat)^*[C(Y_A)]$. Similarly, $f$ is a meet of elements from $\alpha_\flat^*[C(Y_A)]$.
Thus, $f$ is both a pointwise join and meet from $\alpha_\flat^*[C(Y_A)]$.
The argument of \cite[Lem.~2.8]{BMO18c} shows that each element of $C(Y_A)$ is both a pointwise join and meet from $\zeta_A[A]$.
Each $k \in S$ then can be written, in $B(Y_A)$, as $k = \bigvee \{ \zeta_A(a) \mid a \in T_k\}$ for some $T_k \subseteq A$. Set
$T = \bigcup \{ T_k \mid k \in S \}$. We have
\begin{align*}
f &= \bigvee \{ \alpha_\flat^*(k) \mid k \in S\} = \bigvee \{ \alpha_\flat^*\left(\bigvee \zeta_A(a)) \mid a \in T_k\}\right) \mid k \in S \} \\
&= \bigvee \{ \alpha_\flat^*(\zeta_A(a))  \mid a \in T \}.
\end{align*}
Therefore,
$f$ is a pointwise join from $\alpha_\flat^*\zeta_A[A]$. Repeating the argument above but replacing $g$ by $h$ and joins with meets shows that
$f$ is a pointwise meet from $\alpha_\flat^*\zeta_A[A]$.
\end{proof}

As an immediate consequence of Lemma~\ref{lem: pointwise join-meet} we obtain:

\begin{lemma} \label{cor: pointwise join-meet}
Let $e : X \to Y$ be a compactification. Then each $f \in C^*(X)$ is a pointwise join and
meet from $e^*[C(Y)]$.
\end{lemma}

\begin{lemma}
If $e : X \to Y$ is a compactification, then $N(X)$ is isomorphic to $N(Y)$ in $\bal$.
\end{lemma}

\begin{proof}
It follows from Lemma~\ref{cor: pointwise join-meet} that $e^*[C(Y)]$ is join-dense and meet-dense in $C^*(X)$.
Consequently, $C(Y)$ and $C^*(X)$ have isomorphic Dedekind completions. Thus, by Remark~\ref{rem: D(A)}(3),
$N(Y)$ and $N(X)$ are isomorphic in $\bal$.
\end{proof}

\begin{remark}
In fact, the restriction of $e^+:B(Y)\to B(X)$ to $N(Y)$ is a well-defined isomorphism of $N(Y)$ and $N(X)$.
Since we do not require this fact in what follows, we omit the proof.
\end{remark}

Let $e : X \to Y$ be a compactification and $\alpha:=e^\flat:C(Y)\to B(X)$ the corresponding basic extension.
Using \cite[Lem.~4.1]{Dil50} as motivation, we define
$u_\alpha,l_\alpha:B(X)\to B(X)$ as follows. For $f \in B(X)$ let
\[
u_\alpha(f) = \bigwedge \{ \alpha(g) \mid g \in C(Y), f \le \alpha(g)\} \quad \mbox{and} \quad
l_\alpha(f) = \bigvee \{ \alpha(g) \mid g \in C(Y), \alpha(g) \le f\}.
\]
We set
\[
N_\alpha(X) = \{ f\in B(X) \mid l_\alpha u_\alpha(f) = f\}.
\]

\begin{lemma} \label{lem: N_alpha(X)}
Let $e : X \to Y$ be a compactification and $\alpha = e^\flat : C(Y) \to B(X)$ the corresponding basic extension. Then $N_\alpha(X) = N(X)$.
\end{lemma}

\begin{proof}
Let $f \in B(X)$. By \cite[Lem.~4.1]{Dil50}, $f^* = \bigwedge \{ g \in C^*(X) \mid f \le g\}$. By Lemma~\ref{cor: pointwise join-meet}, each
$g\in C^*(X)$ is a pointwise meet from $\alpha[C(Y)]$. Thus, $f^* = u_\alpha(f)$.
A similar argument yields that $f_* = l_\alpha(f)$. Thus, $(f^*)_* = l_\alpha u_\alpha(f)$. From this and the definitions of $N_\alpha(X)$
and $N(X)$ it follows that $N_\alpha(X) = N(X)$.
\end{proof}

This motivates the following definition.

\begin{definition}
Let $\alpha : A \to B$ be a basic extension.
\begin{enumerate}
\item For $b \in B$ set
\[
u_\alpha(b) = \bigwedge \{ \alpha(a) \mid a \in A, b \le \alpha(a) \} \quad \mbox{and} \quad
l_\alpha(b) = \bigvee \{ \alpha(a) \mid a \in A, \alpha(a) \le b\}.
\]
\item Let $N_\alpha = \{ b \in B \mid l_\alpha u_\alpha(b) = b \}$.
\end{enumerate}
\end{definition}

We are ready to define when two basic basic extensions are compatible.

\begin{definition}
We call two basic extensions $\alpha : A \to B$ and $\gamma : C \to B$ \emph{compatible} if $N_\alpha = N_\gamma$.
\end{definition}

\begin{lemma}
\begin{enumerate}
\item[]
\item If $\alpha : A \to B$ is a basic extension, then $\vartheta_B(N_\alpha) = N(X_B)$.
\item Two basic extensions $\alpha : A \to B$ and $\gamma : C \to B$ are compatible iff $\tau_\alpha = \tau_\gamma$.
\end{enumerate}
\end{lemma}

\begin{proof}
(1). By Lemma~\ref{lem: N_alpha(X)}, it suffices to show that $\vartheta_B(N_\alpha) = N_\alpha(X_B)$. We first show that
$\vartheta_B(u_\alpha(b)) = u_\alpha(\vartheta_B(b))$ for each $b \in B$. To see this, since $\vartheta_B$ is an isomorphism
and $\vartheta_B\alpha(a) = \alpha_\flat^*\zeta_A(a)$,
\[
\vartheta_B(u_\alpha(b)) = \bigwedge \{ \vartheta_B(\alpha(a)) \mid a \in A, b \le \alpha(a) \} =
\bigwedge \{ \alpha_\flat^*\zeta_A(a) \mid a \in A, b \le \alpha(a) \}.
\]
On the other hand, $u_\alpha(\vartheta_B(b)) = \bigwedge \{ \alpha_\flat^*(g) \mid g \in C(Y_A), \vartheta_B(b) \le g\}$.
By \cite[Lem.~2.8]{BMO18c}, each $g \in C(Y_A)$ is a pointwise meet from $\zeta_A[A]$. Consequently,
\[
u_\alpha(\vartheta_B(b)) = \bigwedge \{ \alpha_\flat^*\zeta_A(a) \mid a \in A, \vartheta_B(b) \le \alpha_\flat^*\zeta_A(a)\}.
\]
Since $\vartheta_B\alpha(a) = \alpha_\flat^*\zeta_A(a)$ and $\vartheta_B$ is an isomorphism, $b \le \alpha(a)$ iff
$\vartheta_B(b) \le \alpha_\flat^*\zeta_A(a)$. Thus, $\vartheta_B(u_\alpha(b)) = u_\alpha(\vartheta_B(b))$. Similarly,
$\vartheta_B(l_\alpha(b)) = l_\alpha(\vartheta_B(b))$. From this it follows that $\vartheta_B(N_\alpha) = N_\alpha(X_B)$.

(2). Because we are working with two topologies, to avoid confusion, we write $N(X_B, \tau_\alpha)$ and $N(X_B, \tau_\gamma)$.
First suppose that $\tau_\alpha = \tau_\gamma$. Then $N(X_B, \tau_\alpha) = N(X_B, \tau_\gamma)$. Therefore,
by (1), $\vartheta_B(N_\alpha) = \vartheta_B(N_\gamma)$. Since $\vartheta_B$ is 1-1, we conclude that $N_\alpha = N_\gamma$.

Conversely, suppose that $N_\alpha = N_\gamma$. Then $\vartheta_B(N_\alpha) = \vartheta_B(N_\gamma)$, and so
$N(X_B, \tau_\alpha) = N(X_B, \tau_\gamma)$ by (1). To show that $\tau_\alpha = \tau_\gamma$,
it suffices to show that $U \subseteq X_B$ is regular open in $\tau_\alpha$ iff it is regular open in $\tau_\gamma$. Now, $U$ is regular open
in $\tau_\alpha$ iff the characteristic function $\chi_U \in N(X_B, \tau_\alpha)$ (see, e.g., \cite[Ex.~4.11]{BMO16}). The corresponding
statement for $\tau_\gamma$ holds for the same reason. Since $N(X_B, \tau_\alpha) = N(X_B, \tau_\gamma)$, we see that $U$ is regular open in
$\tau_\alpha$ iff $U$ is regular open in $\tau_\gamma$. Thus, $\tau_\alpha = \tau_\gamma$.
\end{proof}

We are now ready to define the notion of a maximal basic extension.

\begin{definition}
\begin{enumerate}
\item[]
\item A basic extension $\alpha : A \to B$ is \emph{maximal} provided that for every compatible extension $\gamma : C \to B$, there is a
morphism $\delta : C \to A$ in $\bal$ such that $\alpha \circ \delta = \gamma$.
\[
\begin{tikzcd}
A \arrow[rr, "\alpha"] && B \\
& C \arrow[ul, "\delta"] \arrow[ur, "\gamma"']
\end{tikzcd}
\]
\item Let $\mbasic$ be the full subcategory of $\basic$ consisting of maximal basic extensions.
\end{enumerate}
\end{definition}

We next give different characterizations of maximal basic extensions.
Let $\alpha : A \to B$ be a basic extension. Then $\alpha_\flat : X_B \to Y_A$ is a continuous map, and so we have a morphism
$(\alpha_\flat)^* : C(Y_A) \to C^*(X_B)$ in $\bal$.

\begin{definition}
Define $\mu:A \to C^*(X_B)$ as the composition $\mu = (\alpha_\flat)^* \circ \zeta_A$.
\end{definition}

Since both $\zeta_A$ and $(\alpha_\flat)^* $ are morphisms in $\bal$, so is $\mu$. In fact,
$\mu$ is a monomorphism in $\bal$.
To see this, note that $\zeta_A$ is 1-1. We show that
$(\alpha_\flat)^*$ is 1-1. If $(\alpha_\flat)^*(f) = 0$ for $f \in C(Y_A)$, then $f \circ \alpha_\flat = 0$.
Therefore, $f|_{\alpha_\flat[X_B]} = 0$. Since $\alpha_\flat[X_B]$ is dense in $Y_A$, we have $f = 0$. Thus,
$(\alpha_\flat)^*$ is 1-1, and so $\mu$ is 1-1.

Let $\iota :  C^*(X_B) \to B(X_B)$ be the
inclusion morphism. The following diagram commutes because the top half commutes by Gelfand-Naimark-Stone duality and the bottom half
commutes by application of the relevant definitions.
\[
\begin{tikzcd}[column sep = 5pc]
A \arrow[r, "\alpha"] \arrow[d, "\zeta_A"] \arrow[dd, bend right = 40, "\mu"'] & B \arrow[d, "\zeta_B"'] \arrow[dd, bend left = 40, "\vartheta_B"] \\
C(Y_A) \arrow[r, "(\alpha_*)^*"] \arrow[d, "(\alpha_\flat)^*"] & C(Y_B) \arrow[d, "\kappa_B"'] \\
C^*(X_B) \arrow[r, "\iota"'] & B(X_B)
\end{tikzcd}
\]

\begin{proposition} \label{maximal}
The following are equivalent for a basic extension $\alpha : A \to B$.
\begin{enumerate}
\item $\alpha$ is maximal.
\item $\mu = (\alpha_\flat)^*\circ\zeta_A : A \rightarrow C^*(X_B)$ is an isomorphism.
\item $A$ is uniformly complete and  $\alpha_\flat : X_B \to Y_A$ is isomorphic to the Stone-\v{C}ech compactification $s : X_B \to \beta X_B$.
\item $A$ is uniformly complete and $\alpha_\flat : X_B \to Y_A$ is equivalent to $s$.
\item The only elements of $B$ that are both a join and meet of elements from $\alpha[A]$ are those that are in $\alpha[A]$.
\end{enumerate}
\end{proposition}

\begin{proof}
(1) $\Rightarrow$ (2). Since $\ubal$ is a reflective subcategory of $\bal$, there is a monomorphism $\widehat{\alpha} : C(Y_A) \to B$ in
$\bal$ with $\widehat{\alpha} \circ \zeta_A = \alpha$. As we pointed out in the proof of Theorem~\ref{prop: uE reflective},
$\widehat{\alpha} : C(Y_A) \to B$ is a basic extension. Since $\vartheta_B$ is an isomorphism (see Proposition~\ref{basic char}),
we may define $\gamma = \vartheta_B^{-1}\circ \iota$. By Example~\ref{ex: basic extensions}(2), $\iota$ is a basic extension.
Thus, $\gamma$ is a basic extension. By Example~\ref{ex: 5.5}, $\tau_\gamma$ is equal to $\tau_\alpha$, and so $\gamma$ is compatible
with $\alpha$. By (1), there is a morphism $\delta : C^*(X_B) \to A$ in $\bal$ such that $\alpha\circ\delta = \gamma$.
\[
\begin{tikzcd}
A \arrow[dd,  "\mu"'] \arrow[rr, "\alpha"] \arrow[rd, "\zeta_A"] && B \arrow[dd, "\vartheta_B"] \\
& C(Y_A) \arrow[ru, "\widehat{\alpha}"] \arrow[ld, "(\alpha_\flat)^*"'] & \\
C^*(X_B) \arrow[uu, bend left = 40,  "\delta"] \arrow[rr, "\iota"'] \arrow[rruu, bend right = 20, "\gamma"'] && B(X_B)
\end{tikzcd}
\]
As we pointed out before the proposition, $\vartheta_B \circ \alpha = \iota \circ \mu$. Therefore,
\[
\iota\circ \mu \circ \delta = \vartheta_B \circ \alpha \circ \delta = \vartheta_B \circ \gamma = \iota.
\]
Since $\iota$ is monic, $\mu\circ \delta$ is the identity on $C^*(X_B)$.
This implies $\mu$ is onto. Because $\mu$ is 1-1, we conclude that $\mu$ is an isomorphism.

(2) $\Rightarrow$ (3). In light of (2), it is clear that $A$ is uniformly complete. Let $f \in C(Y_A)$.
Since the diagram above is commutative,
$(\mu, \vartheta_B)$ is a morphism in $\basic$. Because both $\mu$ and $\vartheta_B$ are isomorphisms, $(\mu, \vartheta_B)$
is an isomorphism in $\basic$. Applying $\sf C$ yields $\alpha_\flat$ and $\iota_\flat$ are isomorphic in $\C$. Therefore,
by Example~\ref{ex: 5.5}, $\alpha_\flat$ and $s$ are isomorphic in $\C$.

(3) $\Rightarrow$ (4). It is proved in \cite[Thm.~3.3]{BMO19a} that if a compactification $e:X\to Y$ is isomorphic to the
Stone-\v{C}ech compactification $s:X\to\beta X$, then $e$ is equivalent to $s$.

(4) $\Rightarrow$ (1). Let $\gamma : C \to B$ be compatible with $\alpha$. Then $\gamma_\flat : X_B \to Y_C$ and $\alpha_\flat : X_B \to Y_A$
are compactifications of the same topological space. By (4), $Y_A$ is homeomorphic to $\beta X_B$, so there is a continuous map
$\varphi : Y_A \to Y_C$ with $\varphi \circ \alpha_\flat = \gamma_\flat$.
\[
\begin{tikzcd}[column sep=5pc]
X_B \arrow[r, "\alpha_\flat"] \arrow[rd, "\gamma_\flat"'] & Y_A \arrow[d, "\varphi"] \\
& Y_C
\end{tikzcd}
\]
This implies that $(\alpha_\flat)^\flat \circ \varphi^* = (\gamma_\flat)^\flat$ since if $f \in C(Y_C)$, then
\[
[(\alpha_\flat)^\flat \circ \varphi^*](f) = f\circ \varphi \circ \alpha_\flat = f \circ \gamma_\flat = (\gamma_\flat)^\flat(f).
\]
Define $\delta = \zeta_A^{-1} \circ \varphi^* \circ \zeta_C$. We have the following diagram.
\[
\begin{tikzcd}[column sep = .1pc]
A \arrow[rr, "\alpha"] && B \\
C(Y_A) \arrow[u, "\zeta_A^{-1}"'] \arrow[rr, "(\alpha_\flat)^\flat"] && B(X_B) \arrow[u, "\vartheta_B^{-1}"] \\
& C(Y_C) \arrow[ul, "\varphi^*"] \arrow[ru, "(\gamma_\flat)^\flat"'] & \\
&& \\
& C \arrow[uu, "\zeta_C"] \arrow[uuuul, bend left = 60, "\delta"] \arrow[uuuur, bend right = 60, "\gamma"'] &
\end{tikzcd}
\]
We just observed that the middle triangle commutes, and the top square commutes by Lemma~\ref{EC}. Another application of Lemma~\ref{EC}
yields that $\gamma = \vartheta_B^{-1} \circ (\gamma_\flat)^\flat \circ \gamma_C$. Thus, $\alpha \circ \delta = \gamma$, which proves
that $\alpha$ is maximal.

(4) $\Rightarrow$ (5). By (4) we may assume $\alpha$ is the basic extension $\iota : C^*(X_B) \to B(X_B)$. Then
$\alpha[A] = C^*(X_B)$. If $b \in B(X_B)$ is a meet from $C^*(X_B)$, then it is upper semicontinuous by \cite[Lem.~4.1]{Dil50},
and if it is a join from $C^*(X_B)$, then it is lower semicontinuous by the dual of \cite[Lem.~4.1]{Dil50}. Therefore, if $b$ is
both a join and meet from $C^*(X_B)$, then $b$ is continuous, so $b \in C^*(X_B) = \alpha[A]$.

(5) $\Rightarrow$ (2).
Let $f \in C^*(X_B)$. By Lemma~\ref{lem: pointwise join-meet}, $f$ is both a pointwise join and a meet from elements
of $(\alpha_\flat)^*[\zeta_A[A]]$.
By (5), $\vartheta_B^{-1}(f) \in \alpha[A]$, so $f \in \vartheta_B\alpha[A] = \mu[A]$. Thus, $\mu$ is onto.
Since it is 1-1, we conclude that $\mu$ is an isomorphism.
\end{proof}

As a consequence, we obtain that $\mbasic$ is a full subcategory of $\ubasic$.

\begin{definition}
Let $\SC$ be the full subcategory of $\C$ consisting of Stone-\v{C}ech compactifications.
\end{definition}

Theorem~\ref{duality cptf} and Proposition~\ref{maximal} immediately yield the following:

\begin{theorem}
There is a dual equivalence between $\SC$ and $\mbasic$.
\end{theorem}

It is well known that $\creg$ and $\SC$ are equivalent (see, e.g., \cite[Prop.~6.8]{BMO19a}).
Thus, as an immediate consequence we obtain:

\begin{theorem} \label{CR duality}
There is a dual equivalence between $\creg$ and $\mbasic$.
\end{theorem}

\begin{remark}
To describe the functors yielding the dual equivalence of Theorem~\ref{CR duality}, we recall that the equivalence between $\creg$ and $\SC$
is obtained by the functors ${\sf S} : \creg \to \SC$ and ${\sf F} : \SC \to \creg$. The covariant functor $\sf S$ sends a completely regular
space $X$ to the Stone-\v{C}ech compactification $s : X \to \beta X$ and a continuous map $f : X \to Y$ to the unique continuous map
$\beta f : \beta X \to \beta Y$ that makes the following diagram commute.
\[
\begin{tikzcd}[column sep=5pc]
X \arrow[r, "f"] \arrow[d, "s_X"'] & Y \arrow[d, "s_Y"] \\
\beta X \arrow[r, "\beta f"'] & \beta Y
\end{tikzcd}
\]
The covariant functor $\sf F$ sends a Stone-\v{C}ech compactification $s : X \to \beta X$ to $X$, and a morphism $(f, \beta f)$ to $f$.
The dual equivalence of Theorem~\ref{CR duality} is obtained by the contravariant functors ${\sf E} \circ {\sf S}$ and ${\sf F} \circ {\sf C}$.
We give a more direct description of the contravariant functors between $\creg$ and $\mbasic$ that yield this dual equivalence.
\[
\begin{tikzcd}[column sep = 5pc, row sep=5pc]
\SC \arrow [r, shift left = .5ex, "\sf E"] \arrow[d, shift left = .5ex, "\sf F"] & \mbasic \arrow[l, shift left = .5ex, "\sf C"]
\arrow[dl, shift left=.5ex, "\sf R"] \\
\creg \arrow[u, shift left = .5ex, "\sf S"] \arrow[ru, shift left = .5ex,  "\sf M"] &
\end{tikzcd}
\]
The contravariant functor $\sf M : \creg \to \mbasic$ sends a completely regular space $X$ to $\iota : C^*(X) \to B(X)$,
and a continuous map $f : X \to Y$ to $(f^*, f^+)$.
\[
\begin{tikzcd}[column sep=5pc]
C^*(Y) \arrow[r, "\iota_Y"] \arrow[d, "f^*"'] & B(Y) \arrow[d, "f^+"] \\
C^*(X) \arrow[r, "\iota_X"'] & B(X)
\end{tikzcd}
\]
By Example~\ref{ex: basic extensions}(2) and Theorem~\ref{maximal}, $\sf{M}$ is a well-defined functor.
The contravariant functor $\sf R$ sends a maximal basic extension $\alpha : A \to B$ to $(X_B, \tau_\alpha)$, and a morphism
$(\rho, \sigma)$ to $\sigma_+ : X_{B'} \to X_B$.
By Theorem~\ref{lem:compactification} and Lemma~\ref{lem: sigma plus}, $\sf{R}$ is a well-defined functor.
That $\sf{F} \circ \sf{C} = \sf{S}$ follows from the definition of the functors, and $\sf{E} \circ \sf{S} \cong \sf{M}$ follows from
Example~\ref{ex: basic extensions}(2). Thus, the above diagram commutes.
\end{remark}

We conclude the article by deriving several consequences of Theorem~\ref{CR duality}.
Recall that a completely regular space $X$ is \emph{strongly zero-dimensional} if
$\beta X$ is zero-dimensional
(see, e.g., \cite[Thms.~6.2.7 and 6.2.12]{Eng89}).
We next obtain a duality for strongly zero-dimensional spaces.

\begin{theorem}
The dual equivalence between $\creg$ and $\mbasic$ restricts to a dual equivalence between the full subcategory of $\creg$ consisting
of strongly zero-dimensional spaces and the full subcategory of $\mbasic$ consisting of the maximal basic extensions $\alpha : A \to B$
for which $A$ is a clean ring.
\end{theorem}

\begin{proof}
Let $X$ be a strongly zero-dimensional space.
Then ${\sf M}(X)$ is the maximal extension $\iota : C^*(X) \to B(X)$.
Since $X$ is strongly zero-dimensional, $C^*(X)$ is clean by \cite[Thm.~2.5]{Aza02}.
Let $\alpha : A \to B$ be a maximal extension with $A$ clean.  Then the image under $\sf R$ of $\alpha$ is the completely
regular space $X_B$. By Proposition~\ref{maximal}, $Y_A$ is the Stone-\v{C}ech compactification of $X_B$. Since $A$ is clean,
$Y_A$ is a zero-dimensional space \cite[Thm.~5.9]{BMO13a}. Thus, $X_B$ is strongly zero-dimensional.
To complete the proof, apply Theorem~\ref{CR duality}.
\end{proof}

\begin{theorem} \label{ED}
The dual equivalence between $\creg$ and $\mbasic$ restricts to a dual equivalence between the full subcategory of $\creg$ consisting of
extremally disconnected spaces and  the full subcategory of $\mbasic$ consisting of the maximal extensions $\alpha : A \to B$
with $A \in \dbal$.
\end{theorem}

\begin{proof}
If $X$ is an extremally disconnected space, then so is $\beta X$
(see, e.g., \cite[Thm.~6.2.27]{Eng89}). The image under $\sf M$ of $X$
is the maximal extension $\iota : C^*(X) \to B(X)$. Since $C^*(X)$ is isomorphic to $C(\beta X)$, we see that
$C^*(X) \in \dbal$ by Corollary~\ref{cor:dbal}.
Conversely, if $\alpha : A \to B$ is a maximal extension, then $A \cong C(Y_A)$, and if $A \in \dbal$, then
$Y_A$ is an extremally disconnected space by Corollary~\ref{cor:dbal}. By Proposition~\ref{maximal}, $Y_A$ is the
Stone-\v{C}ech compactification of $X_B$. Thus, $X_B$ is extremally disconnected (see, e.g., \cite[Thm.~6.2.27]{Eng89}).
Now apply Theorem~\ref{CR duality}.
\end{proof}

Recall that a topological space $X$ is \emph{connected} if $\varnothing,X$ are the only clopens of $X$, and that a commutative ring
$A$ is \emph{indecomposable} if $\Id(A) = \{0,1\}$.

\begin{theorem} \label{Connected}
The dual equivalence between $\creg$ and $\mbasic$ restricts to a dual equivalence between the full subcategory of $\creg$ consisting
of connected spaces and  the full subcategory of $\mbasic$ consisting of the maximal extensions $\alpha : A \to B$ with $A$ an
indecomposable ring.
\end{theorem}

\begin{proof}
Let $X$ be connected. The image under $\sf M$ of $X$ is the maximal extension $\iota : C^*(X) \to B(X)$.
The idempotents of $C^*(X)$ are exactly the characteristic functions of clopen subsets of $X$. Since $X$ is connected, the only clopen subsets are
$\varnothing$ and $X$, so $\Id(C^*(X)) = \{0,1 \}$, and hence $C^*(X)$ is indecomposable. Conversely, if $\alpha : A \to B$ is a
maximal extension with $A$ indecomposable, then $A \cong C(Y_A)$ and $Y_A$ has no nontrivial clopen subsets. Therefore, $Y_A$ is
connected. By Proposition~\ref{maximal}, $Y_A$ is the Stone-\v{C}ech compactification of $X_B$. Thus, $X_B$ is connected (see,
e.g., \cite[Thm.~6.1.14]{Eng89}). Now apply Theorem~\ref{CR duality}.
\end{proof}

\section*{Acknowledgements}
We would like to thank the referee for careful reading and suggesting a more general statement of Lemma~\ref{rem: lifting}.

\def\cprime{$'$}
\providecommand{\bysame}{\leavevmode\hbox to3em{\hrulefill}\thinspace}
\providecommand{\MR}{\relax\ifhmode\unskip\space\fi MR }
\providecommand{\MRhref}[2]{%
  \href{http://www.ams.org/mathscinet-getitem?mr=#1}{#2}
}
\providecommand{\href}[2]{#2}

\end{document}